\documentclass[11pt]{article}

\usepackage[margin=0.9in]{geometry} 
\usepackage{amsmath,amsthm,amssymb,mathrsfs,bbm,titling,mathtools,hyperref,nameref,esint,stmaryrd,enumitem,cite}
\usepackage{color}
\usepackage{fancyhdr}

\let\oldref\ref
\renewcommand{\ref}[1]{(\oldref{#1})}

\newcommand{\N}{\mathbb{N}}
\newcommand{\Z}{\mathbb{Z}}
\newcommand{\C}{\mathbb{C}}
\newcommand{\R}{\mathbb{R}}
\newcommand{\Q}{\mathbb{Q}}

\newcommand{\F}{\mathbb{F}}
\newcommand{\K}{\mathbb{K}}
\newcommand{\eps}{\varepsilon}

\newcommand{\llparen}{(\!(}
\newcommand{\rrparen}{)\!)}

\renewcommand{\O}{\mathcal{O}}

\renewcommand{\cal}{\mathcal}
\newcommand{\bb}{\mathbb}
\renewcommand{\bf}{\mathbf}
\renewcommand{\frak}{\mathfrak}

\newtheorem{theorem}{Theorem}[section]
\newtheorem{lemma}[theorem]{Lemma}

\theoremstyle{definition}
\newtheorem{definition}[theorem]{Definition}

\newtheorem{proposition}[theorem]{Proposition}

\theoremstyle{remark}
\newtheorem{remark}[theorem]{Remark}

\numberwithin{equation}{section}

\title{High-low analysis and small cap decoupling over non-Archimedean local fields}
\author{Ben Johnsrude}

\begin{document}

	\maketitle
	
	\begin{abstract}
		We prove a small cap decoupling theorem for the parabola over a general non-Archimedean local field for which $2\neq 0$. We obtain polylogarithmic dependence on the scale parameter $R$ and polynomial dependence in the residue prime, except for the prime 2 for which the polynomial depends on degree. Our constants are fully explicit.
	\end{abstract}
	
	\section{Introduction}
	
	In this note, we prove that the small cap decoupling theorem for the parabola may be extended to non-Archimedean local fields $\K$ of characteristic different from $2$. We do so by first adapting the ``high-amplitude wave envelope estimate'' of \cite{GM1}. In addition to recovering the desired power law in the scale parameter, we also obtain a fully explicit subpolynomial factor whose scale dependence is of the form $(\log R)^{O(1)}$ with $O(1)$ explicit and not too large, and whose dependence on $\K$ is polynomial in the order of the residue field.

    We recall the standard formalism of small cap decouplings, adapted to the non-Archimedean context. Let $R>1$ be in the range of $|\cdot|=|\cdot|_\K$ (the absolute value on $\K$, normalized to be the modular function), and write $\cal{N}_{\R^{-1}}(\bb{P}^1)$ for the set
    \begin{equation*}
        \big\{(x,y)\in\K^2:|x|\leq 1,\,|y-x^2|\leq R^{-1}\big\}.
    \end{equation*}
    Here we write $\bb{P}^1$ for the truncated parabola $\{(x,x^2)\in\K^2:|x|\leq 1\}$. If $\cal{O}=\{x\in\K:|x|\leq 1\}$ is the closed unit ball, and $\beta\in[\frac{1}{2},1]$ is such that $R^\beta$ also belongs to the range of $|\cdot|_\K$, then we write $\cal{P}(\cal{O},R^{-\beta})$ for the partition of $\cal{O}$ into closed balls of radius $R^{-\beta}$, and for each $I\in\cal{P}(\cal{O},R^{-\beta})$ we write $\gamma_I$ for the set
    \begin{equation*}
        \big\{(x,y)\in\K^2: x\in I,\,|y-x^2|\leq R^{-1}\big\}.
    \end{equation*}
    We will write
    \begin{equation*}
        \Gamma_\beta(R^{-1})=\big\{\gamma_I:I\in\cal{P}(\cal{O},R^{-\beta})\big\};
    \end{equation*}
    thus, $\Gamma_\beta(R^{-1})$ is the partition of $\cal{N}_{R^{-1}}(\bb{P}^1)$ into small caps of dimensions $R^{-\beta}\times R^{-1}$.

    For $p,q\geq 1$ and $R,\beta$ as above, we will write  $D_{p,q}^\K(R;\beta)$ for the infimal constant such that, for any Schwartz--Bruhat function\footnote{i.e. a finite linear combination of indicators of metric balls.} $f:\K^2\to\C$ with Fourier support contained in $\cal{N}_{R^{-1}}(\bb{P}^1)$, we have
    \begin{equation*}
        \|f\|_{L^r(\K^2)}^r\leq D_{r,q}^\K(R;\beta)\left(\sum_{\gamma\in\Gamma_\beta(R^{-1})}\|f_\gamma\|_{L^r(\K^2)}^q\right)^{r/q}.
    \end{equation*}
    Here and elsewhere we will write $f_\gamma$ for the Fourier projection of $f$ onto $\gamma$.

    Our primary goal will be to show the following.
	\begin{theorem}[Small cap decoupling over $\K$]\label{smallcap} Let $r,q\geq 1$ satisfy $\frac{3}{r}+\frac{1}{q}\leq 1$, $R\geq \bf{p}^{32}$, and $\beta\in[\frac{1}{2},1]$. Then the small cap decoupling constant satisfies
		\begin{equation}\label{smd}
			D_{r,q}^\K(R;\beta)\leq 2\frac{\bf q^{-3}\bf{p}^{10}}{(\log\bf p)^{11}}(\log R)^{17+6\beta^{-1}}\Big(R^{\beta(r-\frac{r}{q}-1)-1}+R^{\beta(\frac{r}{2}-\frac{r}{q})}\Big).
		\end{equation}
        Here $\bf q=p^f=\#\bf k$ is the cardinality of the residue field of $\K$, and $\bf{p}=\bf{p}_\K$ is defined as:
        \begin{equation}\label{eq:p_def}
            \bf{p}=\begin{cases}
               \bf q & \K\text{ extends }\Q_p\text{ or }\F_p\llparen t\rrparen,\,p>2\,,\\
               2^d & \K\text{ extends }\Q_2,\,[\K:\Q_2]=d.
            \end{cases}
        \end{equation}
		
	\end{theorem}

    Note in particular that, when $\mathrm{char}(\bf k)\neq 2$, the loss in the size $\bf q=\bf p$ of the residue field is reduced to $\frac{\bf p^7}{(\log\bf p)^{11}}$. The actual inequality we obtain in the proof of Theorem \ref{smallcap} in Section \ref{smallcapproof} is slightly more refined; here we have stated a simplified version.

    The study of non-Archimedean decouplings was initiated in \cite{GLY}, where a bilinear variant of an inequality of the form $D_{r,2}^{\Q_q}(R;\frac{1}{2})\lesssim_{r,\eps}(\log R)^{2r+\eps}$ was established for $q>2$ to achieve good discrete restriction estimates for the parabola. It was continued in \cite{li2022introduction}, where Li observed many of the features of non-Archimedean analysis that made it particularly appropriate for the setting of decoupling. In \cite{johnsrude2024restricted}, the current author and Lin generalized the decoupling theorem for the moment curve to the $q$-adic setting for each $q>n$ (with $n$ the ambient dimension); the result there may be extended to any local field of characteristic either $0$ or greater than $n$.

 Aside from decoupling, there has been a recent flurry of non-Archimedean Fourier and harmonic analysis in general; see e.g. \cite{HW23,wright2020h,FW,krause2022polynomial,hickman2018abstract,hickman2018fourier,arsovski2021p,salvatore2023kakeya} for a short sampling.  

    Over the real numbers, small cap decouplings were introduced in \cite{DGW}, and the estimate $D_{r,r}^\R(R;\beta)\lesssim_\eps R^{\beta(\frac{r}{2}-1)+\eps}$ was proved (with the natural interpretation of $D_{r,q}^\R$). The work of \cite{FGM} established the real version of the estimate we will show, which is a superlevel set estimate implying sharp bounds on the $D_{r,q}^{\square}(R;\beta)$ in the regime $\frac{1}{q}+\frac{3}{r}\leq 1$. Substantial work has also generalized the idea of small cap decouplings to other manifolds, such as that contained in \cite{guth2022small} and \cite{guth2023small}.

    Theorem \ref{smallcap} will proved by the following auxiliary estimate.

	\begin{theorem}[Wave envelope estimate]\label{squarefunction} Let $f:\K^2\to\C$ be Schwartz--Bruhat with Fourier support in $\mathcal{N}_{R^{-1}}(\mathbb{P}^1)$. Then, for any $\alpha>0$,
		\begin{equation*}
			\alpha^4\mu\big(\{x:|f(x)|>\alpha\}\big)\leq \frac{1}{50}\bf q^{-3}\bf{p}^{10}(\log_{\bf p} R)^{10}\sum_{\substack{s\in\bf{p}^\Z\\R^{-1/2}\leq s\leq 1}}\sum_{\substack{\tau\\\mathrm{diam}(\tau)=s}}\sum_{U\in\mathcal{G}_\tau}\mu(U)\left(\fint_U\sum_{\theta\subseteq\tau}|f_\theta|^2\right)^2.
		\end{equation*}
	\end{theorem}
	
	Here we use the following notation: each $U\in\cal{G}_{\tau}$ is a rectangle of dimensions $R\times sR$, with long edge in the direction of the normal vector to $\mathbb{P}^1$ at the center of $\tau$, centered at $0$; the set $\mathcal{G}_\tau$ is the subset of the standard tiling of $\K^2$ by such rectangles for which the following holds:
	\begin{equation}\label{superest}
		\frac{e^2}{2}\frac{(\log R)^2}{(\log\bf{p})^2}\fint_U\sum_{\theta\subseteq\tau}|f_{\theta}|^2\geq\frac{\alpha^2}{(\#\tau)^2}.
	\end{equation}
	Here $\#\tau$ denotes the number of $\tau$ of a particular length for which $f_\tau\not\equiv 0$. We also write $\theta$ for a generic cap on $\bb{P}^1$ of dimensions $R^{-1/2}\times R^{-1}$. The choice of $\cal G_\tau$ in subsection \ref{broad_sub} is more technical, but (by the ``pruning inequalities'') will include strictly fewer envelopes $U$ than in the statement of Theorem \ref{squarefunction}, hence is sufficient.
 
 In \cite{GM1}, these wave envelope estimates were refined to include only those envelopes corresponding to ``high-amplitude'' components of the various square functions. The latter paper demonstrated that the wave envelope estimate could also be used to derive the small cap results of \cite{FGM}. Our argument closely follows our earlier paper \cite{johnsrude2023small}, which in turn was an adaptation of the method in \cite{GM1}.

 We briefly mention the classification of local fields, to supply examples of the fields we will be working over. Any nondiscrete locally compact topological field $\K$ whose topology is induced by an absolute value (i.e. a multiplicative norm), is necessarily one of the following:
 \begin{enumerate}[label=(\alph*)]
     \item $\K=\R,\C$ (i.e. the Archimedean cases);
     \item $\K=\Q_p$, for some prime $p$, or a finite extension thereof; or
     \item $\K=\F_{p^n}\llparen t\rrparen$, for some prime $p$ and natural $n$.
 \end{enumerate}
 We will restrict our attention to cases (b), (c), i.e. the non-Archimedean local fields of characteristic $0$, resp. $p>2$. The reader will generally benefit from imagining the cases $\Q_p,\F_p\llparen t\rrparen$ with $p>2$.

    We also mention the utility in tracking dependence on $\bf{p}$ in the above theorems. Over the reals, Fourier-analytic methods for counting solutions to Diophantine equations frequently entail a subpolynomial loss in the diameter of the variable set. For instance, the Bourgain--Demeter--Guth resolution \cite{bourgain2016proof} of the main conjecture of Vinogradov's mean value theorem shows the particular result
    \begin{equation*}
        J_{s,k}(\cal{A})\lesssim_\eps\mathrm{diam}(\cal{A})^\eps(A^s+A^{2s-\frac{1}{2}k(k+1)}),
    \end{equation*}
    whenever $\cal{A}\subseteq\Z$ has $\#\cal{A}=A$. Here we use the usual notation of
    \begin{equation*}
        J_{s,k}(\cal{A})=\#\Big\{(\bf{n},\bf{m})\in\cal{A}^{2s}:\sum_{\iota=1}^sn_\iota^j-m_\iota^j=0,\,\forall 1\leq j\leq k\Big\},
    \end{equation*}
    for each $s,k\in\N$. If we instead use the $p$-adic decoupling theorem for the moment curve (Theorem 6.1 of \cite{johnsrude2024restricted}), we obtain the alternate estimate
    \begin{equation*}
        J_{s,k}(\cal{A})\lesssim_{p,\eps}\delta_p(\cal{A})^{-\eps}(A^s+A^{2s-\frac{1}{2}k(k+1)}).
    \end{equation*}
    Here $\delta_p(\cal{A})=\min\{|n-m|_p:n\neq m\in\cal{A}\}$, where $|\cdot|_p$ is the usual $p$-adic norm. In particular, if $p$ does not divide any of the differences $n-m$ (say, if $p>\max(|n|:n\in\cal{A})$), then the subpolynomial factor trivializes. However, there are corresponding losses in the choice of prime. Thus, judicious tracking of the dependence on $p$, together with number-theoretic considerations, may reduce the dependence on features of $\cal{A}$ other than its cardinality. See \cite{wooley2023condensation} for another approach to the same problem.

    Next, we record the non-Archimedean version of the ``block example'' of \cite{FGM}, which demonstrates that the estimate in \ref{smallcap} cannot be extended to any $(p,q)$ with $\frac{1}{q}+\frac{3}{p}>1$ and $p>2+2\beta^{-1}$, in the small cap regime $\beta>\frac{1}{2}$. Let $f=f_\theta=\sum_{\gamma\prec\theta}f_\gamma$, where $\theta$ is above $B(0,R^{-1/2})$ and each $f_\gamma=e(c_\gamma\cdot x)1_{B(0,R^{\beta})\times B(0,R)}$; here each $c_\gamma$ is an arbitrary point chosen from $\gamma$. It is quick to see that $\hat{f}_\gamma$ is supported in $\gamma$, for each $\gamma$. Then we have
    \begin{equation*}
        f1_{B(0,R^\beta)\times B(0,R^{2\beta})}=R^{\beta-\frac{1}{2}}1_{B(0,R^\beta)\times B(0,R^{2\beta})},
    \end{equation*}
    so that
    \begin{equation*}
        \|f\|_{L^p}\geq R^{\beta-\frac{1}{2}+\frac{3\beta}{p}},\quad\left(\sum_\gamma\|f_\gamma\|_{L^p}^q\right)^{\frac{1}{q}}=R^{\frac{3\beta}{p}+\frac{1}{q}(\beta-\frac{1}{2})}.
    \end{equation*}
    If $\frac{1}{q}+\frac{3}{p}>1$ and $\beta>\frac{1}{2}$, then the corresponding ratio exceeds $R^{\beta(1-\frac{1}{q}-\frac{1}{p})-\frac{1}{p}}$.

    Lastly, we discuss the special role of the prime $p=2$ in the above; particularly, why we have excluded the fields $\F_{2^d}\llparen t\rrparen$, and why $\bf{p}$ is much larger for extensions of $\Q_2$. In the case of characteristic $2$, say $\K=\F_{2^d}\llparen t\rrparen$, by the Frobenius identity $(x+y)^2=x^2+y^2$ we see that $\bb{P}^1$ is contained in a linear subspace of $\K^2$, when the latter is regarded a vector space over $\F_{p^d}$. In particular, linear equations among the first powers of frequency variables imply corresponding equations among the second powers. A straightforward computation (working first over even integer exponents, then interpolating and comparing with the trivial Cauchy--Schwarz bounds) supplies the identity $D_{p,2}^{\F_{2^d}\llparen t\rrparen}(R;\frac{1}{2})=R^{\frac{1}{2}(\frac{p}{2}-1)}$, for any $R\in 2^{d\N}$ and $p\geq 2$.  One may compare with the proof of the local bilinear restriction estimate Theorem \ref{locbilres} below to see the effect of the size of $2$ in $\K$, when $2\neq 0$ is small.

    \subsection{Applications}\label{subsec:application}

    To motivate the study of non-Archimedean decoupling theorems, we present some applications of Theorem \ref{smallcap}. We emphasize that all of the integral estimates obtained in this section are \emph{real}, i.e. integrals of functions over subsets of $\R^n$. In the terminal section \ref{section:transfer}, we will supply the needed transference procedure.

    Our first application is a variant of the well-known fact that small cap decoupling theorems imply ``short mean-value estimates,'' i.e. mean value estimates along certain subsets of the fundamental domain. We first demonstrated this consequence in \cite{johnsrude2025sparse}. For comparison, we first state the usual result.
    \begin{theorem}[c.f. Corollary 1 of \cite{FGM}]\label{thm:short_mv}
        Let $N\in\N$ and $\sigma\in[0,1].$ Let $a_1,\ldots,a_N\in\C$ be arbitrary. Let $1\leq r,q<\infty$ satisfy $\frac{3}{r}+\frac{1}{q}\leq 1$. Then we have
        \begin{equation}\label{ineq:short_mv}
            \fint_{[0,1]\times[0,N^{\sigma-1}]}\Bigg|\sum_{n=1}^Na_ne\Big(nx+n^2t\Big)\Bigg|^rdxdt\lesssim_\eps N^\eps\big(N^{(r-\frac{r}{q}-1)-1-\sigma}+N^{\frac{r}{2}-\frac{r}{q}}\big)\Bigg(\sum_{n=1}^N|a_n|^q\Bigg)^{r/q}.
        \end{equation}
    \end{theorem}

    Our variant replaces the ``short'' domain $[0,1]\times[0,N^{\sigma-1}]$ with a ``sparse'' domain.
    \begin{definition}
        Let $\sigma,\gamma\in[0,1]$ be two parameters. Let $N\in\N$ be arbitrary. The \emph{$(\sigma,\gamma)$-sparse domain at scale $N$} is the set
        \begin{equation}
            A(N,\sigma;\gamma):=\bigcup_{j=0}^{N^{(1+\sigma)(1-\gamma)}-1}[0,1]\times[jN^{-\sigma(1-\gamma)},jN^{-\sigma(1-\gamma)}+N^{\gamma(1+\sigma)-2}].
        \end{equation}
    \end{definition}
    \begin{remark}
        If $\gamma=1$, the union includes only one term which is simply $[0,1]\times[0,N^{\sigma-1}]$.

        Alternately, if $\gamma=0$, then $A(N,\sigma;1)$ is essentially the vertical $N^{-1}$-neighborhood of the maximal progression of horizontal lines $[0,1]\times\{0\}$, $[0,1]\times\{N^{-\sigma}\},[0,1]\times\{2N^{-\sigma}\},\ldots$, in the range $0\leq t<1$.

        As $\gamma$ ranges from $0$ to $1$, the number of thickened lines decreases and the thickness increases.
    \end{remark}

    \begin{remark}
        It always holds that $\cal{L}^2\big(A(N,\sigma;\gamma)\big)=N^{\sigma-1}$ (as long as $N^{(1+\sigma)(1-\gamma)}$ is an integer). Thus, the $A(N,\sigma;\gamma)$ may be viewed as ``cousins'' of the short domain $[0,1]\times[0,N^{\sigma-1}]$.
    \end{remark}

    With this definition, we may state our first application.
    \begin{theorem}\label{thm:sparse_mv}
        Let $N\in\N$ and $\sigma,\gamma\in[0,1]$. Suppose that $N,N^{(1+\sigma)(1-\gamma)}$ are both powers of a common odd prime $p$. Let $a_1,\ldots,a_N\in\C$ be arbitrary. Let $1\leq r,q<\infty$ satisfy $\frac{3}{r}+\frac{1}{q}\leq 1$. Then we have
        \begin{equation}\label{ineq:sparse_mv}
            \fint_{A(N,\sigma;\gamma)}\Bigg|\sum_{n=1}^Na_ne\Big(nx+n^2t\Big)\Bigg|^rdxdt\lesssim(\log N)^{O(1)}\big(N^{(r-\frac{r}{q}-1)-1-\sigma}+N^{\frac{r}{2}-\frac{r}{q}}\big)\Bigg(\sum_{n=1}^N|a_n|^q\Bigg)^{r/q}.
        \end{equation}
        If $\gamma=0$, the $(\log N)^{O(1)}$ term may be refined to $(\log N)^{23+6\sigma}$.
    \end{theorem}
    Thus, the usual ``short'' (or ``small ball'') mean value estimates \eqref{ineq:short_mv} fit into a family of identical mean value estimates \eqref{ineq:sparse_mv}.

    \begin{remark}
        By a variant of the tensor product argument that will be used to combine real and $p$-adic decoupling for the proof of Theorem \ref{thm:sparse_mv}, it will be possible to somewhat relax the condition that $N,N^{(1+\sigma)(1-\gamma)}$ are powers of a common odd prime. Indeed, by taking tensors over distinct $\Q_p$ with $p$ ranging over a small number of odd primes, we can allow $N,N^{(1+\sigma)(1-\gamma)}$ to have prime factorization composed of a small number of small odd primes. We do not pursue this here.
    \end{remark}

    The proof of Theorem \ref{thm:sparse_mv} involves three steps. First, we prove a variant of Theorem \ref{thm:short_mv} working purely over the $p$-adics for an arbitrary choice of $p$. Second, we demonstrate a ``transference procedure,'' carrying integrals of certain $p$-adic exponential sums to integrals of real exponential sums over a special domain; at this point, we have handled the $\gamma=0$ case of Theorem \ref{thm:sparse_mv}. Lastly, to deal with $0<\gamma<1$, we regard the desired exponential sum as a sum of tensor products of real and $p$-adic exponential sums, and apply the general fact that decoupling theorems tensorize (see e.g. Lemma 5.5 of \cite{johnsrude2024restricted}) to estimate the exponential sum along sets resembling a product of $p$-adic intervals with real intervals.

    A related estimate, Prop. \ref{partial}, which is obtained on the way to proving Theorem \ref{smallcap}, implies the following superlevel set estimate for exponential sums. This is the ``non-Archimedean'' variant of Theorem 4 from \cite{FGM}.
    \begin{theorem}\label{thm:dist}
        Fix $\sigma\in[0,1]$, $\eps>0$, and $p\equiv 3$ $(\mathrm{mod         \ 4})$. Let $N\in\N$ be such that $N,N^{\sigma}$ are both powers of $p$ and such that $p\leq D_\eps N^{\eps/20}$. Let $S\subseteq\{1,\ldots,N\}$. Let $\{a_j\}_{j\in S}$ be a family of complex numbers with $|a_j|\leq 1$.

        Define the distribution functions:
        \begin{equation*}
            \lambda_k(S)=\max_{0\leq a<p^k}\#\big\{j\in S:j\equiv a\quad(\mathrm{mod} \,\,p^k)\big\}.
        \end{equation*}
        Let $f$ be the exponential sum
        \begin{equation*}
            f(x,t)=\sum_{j\in S}a_je\Big(\frac{j}{N}x+\frac{j^2}{N^2}t\Big).
        \end{equation*}
        Let $1\leq\alpha\leq N$ and $U_\alpha=\{(x,t)\in[0,1]^2:|f(x,t)|\geq\alpha\}$. Then we have the superlevel set estimate
        \begin{equation*}
            \frac{\mathcal{L}^2\big(U_\alpha\cap A(N,\sigma;0)\big)}{\mathcal{L}^2\big(A(N,\sigma;0)\big)}\lesssim_\eps p^7\alpha^{-4}D_\eps N^\eps\max_{\frac{1+\sigma}{2}\log_pN\leq k\leq\log_pN}\big(\lambda_{k}(S)\cdot\lambda_{(1+\sigma)\log_pN-k}(S)\big)\sum_{j\in S}|a_j|^4.
        \end{equation*}
    \end{theorem}

    Our third application involves a great many possible examples; we select the simplest one, which we believe adequately illustrates the general theme. In contrast to the two results above, this involves quoting the decoupling theorem over a field extension of some $\Q_p$.

    For $\sigma\in[0,1]$ and $N\in\N$, we write
    \begin{equation*}
        A^{(2)}(N,\sigma)=\bigcup_{j,k=0}^{N^{(1+\sigma)}-1}[0,1]^2\times[jN^{-\sigma-1},jN^{-\sigma-1}+N^{-2}]\times[kN^{-\sigma-1},kN^{-\sigma-1}+N^{-2}].
    \end{equation*}
    \begin{theorem}\label{thm:gaussian_mv}
        Let $N\in\N$ and $\sigma\in[0,1]$. Let $\{a_{nm}\}_{1\leq n,m\leq N}$ be arbitrary complex numbers. Let $1\leq r,q<\infty$ satisfy $\frac{3}{r}+\frac{1}{q}\leq 1$. Then we have
        \begin{equation}
            \begin{split}
            \fint_{A^{(2)}(N,\sigma)}\Bigg|\sum_{n,m=1}^Na_{nm}&e\Big(nx-my+z(n^2-m^2)-2wnm\Big)\Bigg|^rdxdydzdw\\
            &\lesssim(\log N)^{23+6\sigma}\big(N^{(r-\frac{r}{q}-1)-1-\sigma}+N^{\frac{r}{2}-\frac{r}{q}}\big)\left(\sum_{n,m=1}^N|a_{nm}|^q\right)^{r/q}.
            \end{split}
        \end{equation}
    \end{theorem}
    \begin{remark}
        The phase function may be substituted with many equivalents, e.g.
        \begin{equation*}
            nx+my+z(n^2-m^2)+2wmn.
        \end{equation*}
    \end{remark}
    Theorem \ref{thm:gaussian_mv} follows from essentially the same analysis as was used in Theorem \ref{thm:sparse_mv}, with the $\Q_p$ in the latter substituted for $\K=\Q_p\big[\sqrt{-1}\big]$ subject to $p\equiv 3\,\,(\mathrm{mod}\,4)$.

    The proofs of Theorems \ref{thm:sparse_mv}, \ref{thm:dist}, and \ref{thm:gaussian_mv} will be postponed to the end in Section \ref{section:transfer}.

    \subsection{Brief overview of method}\label{overview_sub}

    We discuss the method for proving Theorem \ref{squarefunction}. We will adopt several temporary notations for the sake of an intuitive sketch, which will later be abandoned in favor of the technical approach. 
    
    The basic decomposition used in the proof of Theorem \ref{squarefunction} is the two-part decomposition
    \begin{equation*}
        \cal{N}(\bb{P}^1)=\bigsqcup\theta,\quad \theta-\theta=:\delta\theta=\bigsqcup\delta\theta^{(k)}.
    \end{equation*}
    Here the sets $\delta\theta^{(k)}$ are understood as follows: if $\theta$ is the cap about $0$ for simplicity, that is, $\theta=B(0,R^{-1/2})\times B(0,R^{-1})$, such that in particular $\theta=\delta\theta$ (in our non-Archimedean setting), then for $0\leq k\leq N$ we write $\delta\theta^{(k)}$ for
    \begin{equation*}
        \begin{split}
        \delta\theta^{(k)}&:=\Big\{(x,y)\in\delta\theta:|x|\in\big(R^{\frac{k-1}{2N}-1},R^{\frac{k}{2N}-1}\big]\Big\},\quad 1\leq k\leq N,\\
        \delta\theta^{(0)}&:=\Big\{(x,y)\in\delta\theta:|x|\leq R^{-1}\Big\}.
        \end{split}
    \end{equation*}
    Here $N\sim\log R$ is an integer. Other $\delta\theta$ will be decomposed similarly; so too for caps of different sizes, e.g. the $\tau$ of shape $R^{-1/3}\times R^{-2/3}$. Thus, the parameter $k$ measures the distance from the center of the cap. The raison d'\^etre of this decomposition is the pair of estimates
    \begin{equation}\label{introhigh}
        \int\left|\sum_{\theta}\cal{P}_{\delta\theta^{(k)}}\big[|f|^2\big]\right|^2\lesssim\sum_{\tau}\int\left|\sum_{\theta\subseteq\tau}\cal{P}_{\delta\theta^{(k)}}\big[|f|^2\big]\right|^2,
    \end{equation}
    (writing $\cal{P}_A$ for the Fourier projection onto a set $A$), valid whenever $k>0$ and the $\tau$ have diameter $d(\tau)\geq R^{\frac{k}{2N}-1}$; and the pointwise estimate
    \begin{equation}\label{introlow}
        \left|\sum_\theta \cal{P}_{\delta\theta^{(k)}}[f]\right|^2=\sum_\tau\left|\sum_{\theta\subseteq\tau} \cal{P}_{\delta\theta^{(k)}}[f]\right|^2,
    \end{equation}
    valid whenever $k>0$ and the $\tau$ have diameter $d(\tau)\geq R^{-\frac{|k|}{N}}$. \eqref{introhigh} and \eqref{introlow} are known as the high and low lemmas, respectively; see Lemmas \ref{highlemma} and \ref{lowlemma} below. By repeatedly applying estimates of the form \eqref{introhigh} and \eqref{introlow}, together with the usual $L^4$ C\'ordoba--Fefferman square function estimate, we find that for each $k\geq\ell$ we have that
    \begin{equation*}
        \int\Big|\sum_{\substack{\tau\\\mathrm{diam}(\tau)=R^{-\frac{\ell}{N}}}}\cal{P}_{\delta\tau^{(k)}}[f]\Big|^4\lessapprox\sum_{m=\ell}^N\sum_{\substack{\tau\\\mathrm{diam}(\tau)=R^{-\frac{m}{N}}}}\int\Big|\sum_{\theta\subseteq\tau}\cal{P}_{\delta\theta^{(m)}}\big[|f|^2\big]\Big|^2.
    \end{equation*}
    See the proof of Prop. \ref{lasthighest} below. It remains to analyze the right-hand side. One may observe that each $\Big|\sum_{\theta\subseteq\tau}\cal{P}_{\delta\theta^{(k)}}\big[|f|^2\big]\Big|^2$ is constant on rectangles $U$ of dimension $R\times R^{1-\frac{m}{N}}$, oriented as stated in Theorem \ref{squarefunction}. Thus, for each $U$,
    \begin{equation*}
        \int_U\Big|\sum_{\theta\subseteq\tau}\cal{P}_{\delta\theta^{(k)}}\big[|f|^2\big]\Big|^2\leq\mu(U)\left(\fint_U\sum_{\theta\subseteq\tau}\big|\cal{P}_\theta[f]\big|^2\right)^2.
    \end{equation*}
    Thus, we are done in the special case that (a) there is some $\ell$ and $k\geq\ell$ such that $f=\sum_{\tau:\mathrm{diam}(\tau)=R^{-\ell/N}}\cal{P}_{\tau^{(k)}}[f]$, and that (b) for each $m\geq\ell$ and each $\tau$ with $\mathrm{diam}(\tau)=R^{-m/N}$ we have that $\sum_{\theta\subseteq\tau}\cal{P}_{\delta\theta^{(k)}}\big[|f|^2\big]$ is supported on the rectangles in $\cal{G}_{\tau}$. It happens that this special case may be achieved from the general one by a pruning procedure on $f$; that is, an arbitrary input function $f$ is a sum of functions $f_m^\cal{B}$ satisfying (a) and (b), plus an inanity $f_0$ which is small for trivial reasons.

    The argument we follow will put in front the pruned functions $f_m^\cal{B}$, and the Fourier decomposition we sketched above will be expressed in somewhat different language (i.e. the high/low analysis of square functions). In particular, the decompositions $\delta\theta=\bigsqcup\delta\theta^{(k)}$ discussed above are not referenced past this point.
    
    The principal thrust of the argument is equivalent to that of \cite{GM1} and \cite{johnsrude2023small}, though translated to the non-Archimedean setting. This latter stipulation provides many technical advantages, particularly related to the absence of Schwartz tails. The most visible consequence of this is the removal of many technical weights that were present in earlier papers. As a corollary, our analysis is relatively simple, and may be used as a comparative document for those wishing to study the earlier works.
	
	\subsection{Initial notation-setting}

    Let $\K$ be a non-Archimedean local field of characteristic not $2$, i.e. a nondiscrete totally disconnected locally compact topological field which is equipped with a complete absolute value $|\cdot|$ inducing the topology, for which $2\neq 0$. We normalize $|\cdot|$ by insisting that it is the modular function for $\K$, regarded as a LCA group.
    
    Let $\mathcal{V}\subseteq(0,\infty)$ be the range of $|\cdot|$ on the nonzero members of $\K$. Let $\bf{m}\subseteq\cal{O}$ be the maximal ideal, and $\varpi\in\bf{m}$ be a uniformizer. Let $\eta\in\N$ be minimal such that $\mathbf{p}=|\varpi|^{-\eta}$ satisfies $\bf{p}\geq|2|^{-1}$; note that this $\bf p$ matches with the choice in \eqref{eq:p_def}. We will also write $\bf q=\#\bf k=p^f$ for the cardinality of the residue field $\bf k$ of $\K$. Note in particular that $\bf q=\bf p$ if $\mathrm{char}(\bf k)\neq 2$, and $\bf q\leq\bf p=|2|^{-1}$ if $\mathrm{char}(\bf k)=2$.
    
    Fix some $R\in\mathbf{p}^{2\N}$, and write $N=\frac{1}{2}\log_{\bf{p}}(R)$.  For $0\leq k\leq N$, we write $R_k=\bf{p}^{k}$. Let $\alpha\in(0,R)$ and $U_\alpha=\{x\in B_R:|f(x)|>\alpha\}$.

    Write $B_T=\{x\in\K^2:|x|\leq T\}$ for each $T\in\R_{>0}$. For each $k$, write $\mathcal{P}(\cal{O},R_k^{-1})$ for the partition of the unit ball $B_1$ into metric balls of radius $R_k^{-1}$. If $I\in\cal{P}(\cal{O},R_k^{-1})$, we will write $\tau_I$ for the set
    \begin{equation*}
        \tau_I=(a,a^2)+M_{a,\varpi^{\eta k}}[\cal{O}^2],
    \end{equation*}
    where $a\in I$ is chosen arbitrarily and, for each $\lambda\in\K^\times$,
    \begin{equation*}
        M_{a,\lambda}=\begin{bmatrix}
            1 & 0 \\ 2a & 1
        \end{bmatrix}.\mathrm{diag}(\lambda,\lambda^2).
    \end{equation*}
    It is quick to see that $\tau_I$ is independent of $a\in I$. We will write $A_{a,\lambda}(x)=(a,a^2)+M_{a,\lambda}(x)$. We will also lift the subset partial order $\subseteq$ on the set of metric balls $I$ to the associated caps $\tau_I$, which we will write as $\prec$. Thus,
    \begin{equation*}
        I\subseteq J\iff\tau_I\prec\tau_J.
    \end{equation*}
    We introduce this special notation $\prec$ to help clarify at various points that the two objects on either side of the relation symbol are objects of a special type (i.e. caps), rather than arbitrary subsets of $\K^2$.

    Let $\mu$ denote Haar measure on $\K^1$ and $\K^2$; the choice will always be clear from context. We will let $e:\K\to\C$ be some choice of character such that $e(\O)=\{1\}\neq e(\varpi^{-1}\O)$. A convenient choice for $\Q_p$ and $\F_p\llparen t\rrparen$ would be
    \begin{equation}\label{id:padic_char}
        e\left(\sum_{n=N}^\infty a_np^n\right)=\exp\left(2\pi i \sum_{n=N}^{-1}a_np^n\right),\quad N\in\Z, \{a_n\}_n\in\{0,\ldots,p-1\}^\Z
    \end{equation}
    for $\Q_p$, and
    \begin{equation*}
        e\left(\sum_{n=N}^\infty a_nt^n\right)=\exp(2\pi i a_{-1}p^{-1}),\quad N\in\Z, \{a_n\}_n\in\{0,\ldots,p-1\}^\Z
    \end{equation*}
    for $\F_p\llparen t\rrparen$.
    
    With respect to $e$ and $\mu$, we understand the Fourier transform to be
    \begin{equation*}
        \hat{f}(\xi)=\int_\K e(x\xi)f(x)d\mu(x),
    \end{equation*}
    for any Schwartz--Bruhat function $f:\K\to\C$. Functions on $\K^n$ will be handled similarly.

    For $I\in\cal{P}(\cal{O},R_k^{-1})$, we will write
    \begin{equation*}
        N_{\tau_I}=M_{a,1}^{-\top}.\mathrm{diag}(\varpi^{-\eta(N-k)},\varpi^{-
        \eta N}),
    \end{equation*}
    where we make arbitrary choices of $a\in I$. Let $\cal{U}_{\tau_I}$ be the image of the set of translates of $\cal{O}^2$ under $N_{\tau_I}$; $\cal{U}_{\tau_I}$ does not depend on the choice of $a$. For $U\in\cal{U}_{\tau_I}$, define the averaging operator $\cal{A}_U$ by
    \begin{equation*}
        \cal{A}_U[f]=\mu(U)^{-1}\int_Ufd\mu.
    \end{equation*}

    We will use the symbol $\chi$ for the indicator functions of annuli, and put the radius bounds in the denominator. Thus,
    \begin{equation*}
        \chi_{\leq r}=1_{B_{r}},\quad \chi_{>r}= 1_{\K^2\setminus B_{r}},\quad \chi_{(r_1,r_2]}=1_{B_{r_2}\setminus B_{r_1}}.
    \end{equation*}

    We will use subscripts to denote Fourier projections, e.g. $\cal{P}_\theta[f]=f_\theta$, and hereafter omit the projection operators $\cal{P}_A$ discussed exclusively in subseection \ref{overview_sub}. We will write $g_\tau=\sum_{\theta\prec\tau}|f_\theta|^2$.

    \subsection{Overview of remainder of paper}

    In Section \ref{sqfunctionproof}, we prove Theorem \ref{squarefunction}. In the first subsection \ref{lemmas_sub}, we state and prove the basic lemmas that apply to general functions of the prescribed spectral support, which will be the foundation of our analysis. In the second subsection \ref{broad_sub}, we fix the particular function $f$ and define a pruning of that function, to define a decomposition $f=\sum_{m}f_m^\cal{B}+(f-f_N)+f_0$ into pieces over which the preceding lemmas may be applied fruitfully. The upshot of that subsection is a set of estimates that will resolve Theorem \ref{squarefunction} in the special case of ``broad'' domination. In the next subsection \ref{bn_sub}, we run a broad/narrow analysis to conclude a local version of Theorem \ref{squarefunction}. In the terminal subsection \ref{local_sub}, we prove the full theorem by removing the local assumption. In Section \ref{smallcapproof}, we prove Theorem \ref{smallcap} by an essentially elementary manipulation of the conclusion of Theorem \ref{squarefunction}.

	\section{Proof of Theorem \ref{squarefunction}}\label{sqfunctionproof}

    We begin by establishing a number of technical high/low decomposition results, applicable to a general Schwartz--Bruhat function $f$ with Fourier support in $\cal{N}_{R^{-1}}(\bb{P}^1)$. Later, we will fix a single $f$ and perform a pruning procedure, in order to obtain appropriate functions for which the preceding results are useful. 

    \subsection{Technical lemmas}\label{lemmas_sub}

	\begin{lemma}[Low lemma]\label{lowlemma}
		Let $f$ have Fourier support in $\cal{N}_{R^{-1}}(\bb{P}^1)$. For any $1\leq m\leq k\leq N$, and $0\leq s\leq k$,
		\begin{equation*}
			|f|^2*\chi_{\leq R_k^{-1}}^\vee=\sum_{\tau_k}|f|^2*\chi_{\leq R_k^{-1}}^\vee
		\end{equation*}
		for any $\tau_s$.
	\end{lemma}
	\begin{proof}
        Indeed, if $\tau_k\neq\tau_k'$, then for each $x\in\tau_k$ and $y\in\tau_k'$, $|x-y|>R_k^{-1}$. Thus, $\tau_k-\tau_k'$ is disjoint from $B_{R_k^{-1}}$.
	\end{proof}
	
	\begin{lemma}[High Lemmas]\label{highlemma} Let $f$ have Fourier support in $\cal{N}_{R^{-1}}(\bb{P}^1)$. For any $m,k$, and $l$ such that $1\leq m\leq N$, $l\leq k$, and $k+l\leq N$,
		\begin{enumerate}
			\item[(a)] 
			\begin{equation*}
				\int\Big|\sum_\theta|f_{\theta}|^2*\chi_{>R_k/R}^\vee\Big|^2\leq\sum_{\tau_k}\int\Big|\sum_{\theta\prec\tau_k}|f_{\theta}|^2*\chi_{>R_k/R}^\vee\Big|^2,
			\end{equation*}
			\item[(b)]
			\begin{equation*}
				\int\Big|\sum_{\tau_{k}}|f_{\tau_{k}}|^2*\chi_{>R_{k+l}^{-1}}^\vee\Big|^2\leq|2|^{-1}\bf{p}^{l}p^{-f}\sum_{\tau_k}\int\Big||f_{\tau_k}|^2*\chi_{>R_{k+l}^{-1}}^\vee\Big|^2.
			\end{equation*}
		\end{enumerate}
	\end{lemma}
	\begin{proof}
		(a): By Plancherel,
		\begin{equation*}
			\int\Big|\sum_\theta|f_{\theta}|^2*\chi_{>R_k/R}^\vee\Big|^2=\int_{|\xi|>R_k/R}\Big|\sum_{\tau_k}\sum_{\theta\prec\tau_k}\widehat{|f_{\theta}|^2}(\xi)\Big|^2.
		\end{equation*}
		The supports of the summands $\sum_{\theta\prec\tau_k}\widehat{|f_{m,\theta}^{\mathcal{B}}|^2}$, ranging over distinct $\tau_k$, are disjoint outside of the ball $B_{R_k/R}$. Applying Plancherel, we conclude.
		
		(b): Note that $|f_{\tau_k}|^2$ has Fourier support in the set $\tau_k-\tau_k$. Suppose that $\tau_k$ is centered at $\gamma(t_1)$ and $\tau_k'$ is centered at $\gamma(t_2)$, for some $|t_1-t_2|\geq |2|^{-1}\bf{p}^{l-k}$. Then, if $(\tau_k-\tau_k)\cap(\tau_k'-\tau_k')\setminus B_{R_{k+l}^{-1}}$ is nontrivial, then we may find a solution $a_1,a_2,b_1,b_2\in\cal{O}$ to the system
        \begin{equation*}
            |a_1|,|a_2|>\bf{p}^{k-(k+l)},
        \end{equation*}
        \begin{equation*}
            (\varpi^{\eta k}(a_1-a_2), \varpi^{2\eta k}(b_1-b_2)+2\varpi^{\eta k}(t_1a_1-t_2a_2))=(0,0).
        \end{equation*}
        By the first condition on the second display, $a_1=a_2$. But then
        \begin{equation*}
            |\varpi^{2\eta k}(b_1-b_2)+2\varpi^{\eta k}(t_1a_1-t_2a_2)|\geq |2|\bf{p}^{-k}|a_1||t_1-t_2|-\bf{p}^{-2k}|b_1-b_2|>\bf{p}^{-(-k+k-k-l+l-k)}-\bf{p}^{-2k}=0.
        \end{equation*}
        By the strict inequality, the leftmost expression is nonzero, contradicting the previous assummption.

        Thus, we know that the $\tau_k-\tau_k$ can only overlap on $\Q_p^2\setminus B_{R_{k+l}^{-1}}$ with those $\tau_k'-\tau_k'$ corresponding to time parameters within a distance of $|2|^{-1}\bf{p}^{l-k}$. Thus, appealing to Plancherel,
        \begin{equation*}
            \begin{split}
            \int\Big|\sum_{\tau_{k}}|f_{\tau_{k}}|^2*\chi_{>R_{k+l}^{-1}}^\vee\Big|^2&=\int_{\Q_p^2\setminus B_{R_{k+l}^{-1}}}\Big|\sum_{\tau_{k}}\widehat{|f_{\tau_{k}}|^2}\Big|^2\\
            &\leq|2|^{-1}\bf{p}^{l}p^{-f}\sum_{\tau_k}\int\Big||f_{\tau_k}|^2*\chi_{>R_{k+l}^{-1}}^\vee\Big|^2
            \end{split}
        \end{equation*}
        by the Schur test.
		
	\end{proof}

    \begin{proposition}[Wave envelope bound of high parts]\label{lasthighest} Let $f$ have Fourier support in $\cal{N}_{R^{-1}}(\bb{P}^1)$ and $1\leq\ell<N$. Then
        \begin{equation*}
            \int\left|\sum_{\tau_{\ell}}|f_{\tau_{\ell}}|^2*\chi_{> R_{\ell}/R}^\vee\right|^2\leq 3\bf{q}^{-1}\bf{p}^{3}N\sum_{k=\ell+1}^N\sum_{\tau_k}\int\Big|\sum_{\theta\prec\tau_k}|f_\theta|^2*\chi_{\leq R_k/R}\Big|^2.
        \end{equation*}

    \end{proposition}

    \begin{proof}
        We decompose
        \begin{equation*}
            \begin{split}
            \int\left|\sum_{\tau_{\ell}}|f_{\tau_{\ell}}|^2*\chi_{> R_{\ell}/R}^\vee\right|^2=&\sum_{k=\ell+1}^N\int\left|\sum_{\tau_{\ell}}|f_{\tau_{\ell}}|^2*\chi_{(R_{k-1}/R,R_k/R]}^\vee\right|^2\\
            &+\sum_{k=\ell+1}^N\int\left|\sum_{\tau_{\ell}}|f_{\tau_{\ell}}|^2*\chi_{(R_{k}^{-1},R_{k-1}^{-1}]}^\vee\right|^2.
            \end{split}
        \end{equation*}
        
		Consider the first sum. By the low lemma \ref{lowlemma},
		\begin{equation*}
			\int\Big|\sum_{\tau_{\ell}}|f_{\tau_{\ell}}|^2*\chi_{(R_{k-1}/R,R_k/R]}^\vee\Big|^2=\int\Big|\sum_{\theta}|f_{\theta}|^2*\chi_{(R_{k-1}/R,R_k/R]}^\vee\Big|^2.
		\end{equation*}
		By the high lemma (a) and Cauchy--Schwarz, we have
        \begin{equation*}
			\int\Big|\sum_{\theta}|f_{\theta}|^2*\chi_{(R_{k-1}/R,R_k/R]}^\vee\Big|^2\leq\bf{p}\sum_{\tau_{k}}\int\Big|\sum_{\theta\prec\tau_{k}}|f_{\theta}|^2*\chi_{\leq R_k/R}^\vee\Big|^2.
		\end{equation*}
		
		Next, we consider the second sum in our decomposition. For each $\ell+1\leq k\leq N$, by the low lemma \ref{lowlemma},
        \begin{equation*}
			\int\Big|\sum_{\tau_{\ell}}|f_{\tau_{\ell}}|^2*\chi_{(R_{k}^{-1},R_{k-1}^{-1}]}^\vee\Big|^2=\int\Big|\sum_{\tau_{k-1}}|f_{\tau_{k-1}}|^2*\chi_{(R_{k}^{-1},R_{k-1}^{-1}]}^\vee\Big|^2.
		\end{equation*}
		By part (b) of the high lemma \ref{highlemma},
		\begin{equation*}
            \int\Big|\sum_{\tau_{k-1}}|f_{\tau_{k-1}}|^2*\chi_{(R_{k}^{-1},R_{k-1}^{-1}]}^\vee\Big|^2\leq|2|^{-1}\bf q^{-1}\bf{p}\sum_{\tau_{k-1}}\int|f_{\tau_{k-1}}|^4.
		\end{equation*}
		By the reverse square function estimate for $\mathbb{P}^1$, Prop. \ref{cord_feff} below,
		\begin{equation*}
			\int|f_{\tau_{k-1}}|^4\leq 2\int\Big|\sum_{\theta\prec\tau_{k-1}}|f_{\theta}|^2\Big|^2.
		\end{equation*}
		We decompose the right-hand side of the resulting inequality:
		\begin{equation*}
            \begin{split}
			\int\Big|\sum_{\tau_{k-1}}|f_{\tau_{k-1}}|^2*\chi_{(R_{k}^{-1},R_{k-1}^{-1}]}^\vee\Big|^2&\leq 2|2|^{-1}\bf q^{-1}\bf{p}\sum_{\tau_{k-1}}\int\Big|\sum_{\theta\prec\tau_{k-1}}|f_{\theta}|^2\Big|^2\\
            &=2|2|^{-1}\bf q^{-1}\bf{p}\sum_{\tau_{k-1}}\int\Big|\sum_{\theta\prec\tau_{k-1}}|f_{\theta}|^2*\chi_{\leq R_{k-1}/R}^\vee\Big|^2\\
            &+2|2|^{-1}\bf q^{-1}\bf{p}\sum_{k\leq t\leq N}\sum_{\tau_{k-1}}\int\Big|\sum_{\theta\prec\tau_{k-1}}|f_{\theta}|^2*\chi_{(R_{t-1}/R,R_{t}/R]}^\vee\Big|^2.
            \end{split}
		\end{equation*}
        The first summand is of the desired form. Take now $k\leq t\leq N$. By the high lemma,
        \begin{equation*}
            \sum_{\tau_{k-1}}\int\Big|\sum_{\theta\prec\tau_{k-1}}|f_{\theta}|^2*\chi_{(R_{t-1}/R,R_{t}/R]}^\vee\Big|^2\leq\sum_{\tau_{t-1}}\int\Big|\sum_{\theta\prec\tau_{t-1}}|f_{\theta}|^2*\chi_{(R_{t-1}/R,R_{t}/R]}^\vee\Big|^2.
        \end{equation*}
        By Cauchy--Schwarz, we have
        \begin{equation*}
            \sum_{\tau_{t-1}}\int\Big|\sum_{\theta\prec\tau_{t-1}}|f_{\theta}|^2*\chi_{(R_{t-1}/R,R_{t}/R]}^\vee\Big|^2\leq\bf{p}\sum_{\tau_{t}}\int\Big|\sum_{\theta\prec\tau_{t}}|f_{\theta}|^2*\chi_{\leq R_{t}/R}^\vee\Big|^2.
        \end{equation*}
        
        Adding our two contributions, we obtain
        \begin{equation*}
            \int\left|\sum_{\tau_{\ell}}|f_{\tau_{\ell}}|^2*\chi_{> R_{\ell}/R}^\vee\right|^2\leq\sum_{k=\ell+1}^N\Big(\frac{2\bf p^2}{|2|\bf q}(k-\ell)+\bf p\Big)\sum_{\tau_k}\int\Big|\sum_{\theta\prec\tau_k}|f_\theta|^2*\chi_{\leq R_k/R}\Big|^2.
        \end{equation*}
  
    \end{proof}

    We conclude this subsection by recording the C\'ordoba--Fefferman argument over $\K$, which is here represented by two results: Theorem \ref{locbilres} (the local form) and Prop. \ref{cord_feff} (the global form).

    \begin{theorem}[Local bilinear restriction]\label{locbilres} Suppose $|2|=\epsilon\in(0,1]$ in $\K$. Let 
    \begin{equation*}
        0<S\leq D\leq \epsilon\Gamma\leq\epsilon,\quad R>0.
    \end{equation*}
    Let $J\in\cal{P}(\cal{O},\Gamma)$ and $L,L'\in\cal{P}(J,\epsilon^{-1}D)$ be distinct, $I,I'\in\cal{P}(L,D),\cal{P}(L',D)$, respectively, and write $\tau=\tau_I,\tau'=\tau_{I'}$ for the caps above $I,I'$ on $\mathbb{P}^1$. Then we have the estimate
    \begin{equation*}
        \int_{B_{R}}|f_\tau|^2|f_{\tau'}|^2\leq \max(\bf q^{-2}(\Gamma/D)^2(SR)^{-2},1)\int_{B_{R}}\left(\sum_{K\in\cal{P}(J,S)}|f_{\tau_K}|^2\right)^2.
    \end{equation*}
        
    \end{theorem}
    \begin{proof}
        We may assume that $J=B(0,\Gamma)$. We will use $\theta,\vartheta$ for various $\tau_K$ with $K\in\cal{P}(J,S)$. Computing directly,
        \begin{equation*}
            \int_{B_R}|f_\tau|^2|f_{\tau'}|^2=\sum_{\substack{\theta_1,\theta_2\prec\tau\\\vartheta_1,\vartheta_2\prec\tau'}}\int_{B_R}f_{\theta_1}\overline{f}_{\theta_2}f_{\vartheta_1}\overline{f}_{\vartheta_2},
        \end{equation*}
        where a given tuple $(\theta_1,\theta_2,\vartheta_1,\vartheta_2)$ produces a nontrivial summand only if $(\theta_1-\theta_2+\vartheta_1-\vartheta_2)\cap B_{R^{-1}}\neq\emptyset$. Suppose this is the case. We may find $\xi_1,\xi_2,\eta_1,\eta_2$ be such that
        \begin{equation*}
            (\xi_1,\xi_1^2)\in\theta_1,(\xi_2,\xi_2^2)\in\theta_2,\quad(\eta_1,\eta_1^2)\in\vartheta_1,(\eta_2,\eta_2^2)\in\vartheta_2,
        \end{equation*}
        and such that there are $\eps_1,\eps_2,\delta_1,\delta_2\in\K$ with $|\eps_\iota|,|\delta_\iota|\leq S^2$, such that
        \begin{equation*}
        (\xi_1,\xi_1^2+\eps_1)-(\xi_2,\xi_2^2+\eps_2)+(\eta_1,\eta_1^2+\delta_1)-(\eta_2,\eta_2^2+\delta_2)\in B_{R^{-1}}.
        \end{equation*}
        
        Note that
        \begin{equation*}
            |(\xi_1+\xi_2)-(\eta_1+\eta_2)|>D;
        \end{equation*}
        indeed, the two bracketed expressions belong to distinct members of $\cal{P}(\cal{O},|2|\epsilon^{-1} D)$. By the assumption on the points, $|(\xi_1-\xi_2)+(\eta_1-\eta_2)|\leq R^{-1}$. Thus,
        \begin{equation*}
            \xi_1^2-\xi_2^2+\eta_1^2-\eta_2^2=(\xi_1-\xi_2)((\xi_1+\xi_2)-(\eta_1+\eta_2))+\cal{O}(\Gamma R^{-1}).
        \end{equation*}
        The first summand has size $>D|\xi_1-\xi_2|$, so the vanishing condition may hold only if
        \begin{equation*}
            |\xi_1-\xi_2|< D^{-1}\max(\Gamma R^{-1},S^2).
        \end{equation*}
        Symmetrically, it is necessary for $|\eta_1-\eta_2|<T:=D^{-1}\max(\Gamma R^{-1},S^2)$.  We divide into cases. Suppose first that $T\leq \bf qS$. Then we certainly have
        \begin{equation*}
            \int_{B_R}|f_\tau|^2|f_{\tau'}|^2=\sum_{K_1\in\cal{P}(I,S)}\sum_{K_2\in\cal P(I',S)}\int_{B_R}|f_{\tau_{K_1}}|^2|f_{\tau_{K_2}}|^2,
        \end{equation*}
        and we are done.
        
        In the alternate case, we have $S<\bf q^{-1}T$, and we take our weaker identity
        \begin{equation*}
            \int_{B_R}|f_\tau|^2|f_{\tau'}|^2=\sum_{J_1\in\cal{P}(I,\bf q^{-1}T)}\sum_{J_2\in\cal P(I',\bf q^{-1}T)}\int_{B_R}|f_{\tau_{J_1}}|^2|f_{\tau_{J_2}}|^2,
        \end{equation*}
        and apply Cauchy--Schwarz to conclude
        \begin{equation*}
            \int_{B_R}|f_\tau|^2|f_{\tau'}|^2\leq \bf q^{-2}S^{-2}T^2\sum_{K_1\in\cal{P}(I,T)}\sum_{K_2\in\cal P(I',T)}\int_{B_R}|f_{\tau_{K_1}}|^2|f_{\tau_{K_2}}|^2.
        \end{equation*}
        Finally, we have the implication
        \begin{equation*}
            S<\bf q^{-1}T\implies S^2\leq\Gamma R^{-1},
        \end{equation*}
        so that our inequalities may take the shape
        \begin{equation*}
            \int_{B_R}|f_\tau|^2|f_{\tau'}|^2\leq\int_{B_R}\left(\sum_{K\in\cal{P}(J,S)}|f_{\tau_K}|^2\right)^2\times\begin{cases}
                1 & T\leq \bf qS\\
                \bf q^{-2}(\Gamma/D)^2(RS)^{-2} & S<\bf q^{-1}T,
            \end{cases}
        \end{equation*}
        which clearly implies our conclusion.
        
    \end{proof}

    Similarly, we may prove:
    \begin{proposition}[C\'ordoba--Fefferman square function estimate]\label{cord_feff} Suppose $\mathrm{char}(\K)\neq 2$. If $f$ has Fourier support in $\cal{N}_\delta(\mathbb{P}^1)\subseteq\K^2$, then 
        \begin{equation*}
            \int|f|^4\leq 2\int\left(\sum_\theta|f_\theta|^2\right)^2,
        \end{equation*}
        where the caps $\theta$ have size $\delta\times\delta^2$.
    \end{proposition}
    \begin{proof}
        We expand out
        \begin{equation*}
            \int|f|^4=\sum_{\theta_1,\theta_2,\vartheta_1,\vartheta_2}\int f_{\theta_1}\overline{f}_{\theta_2}f_{\vartheta_1}\overline{f}_{\vartheta_2}.
        \end{equation*}
        By essentially the same analysis above, each tuple corresponds to a nontrivial integral only if $\theta_1=\theta_2,\vartheta_1=\vartheta_2$ or $\theta_1=\vartheta_2,\theta_2=\vartheta_1$. Thus,
        \begin{equation*}
            \sum_{\theta_1,\theta_2,\vartheta_1,\vartheta_2}\int f_{\theta_1}\overline{f}_{\theta_2}f_{\vartheta_1}\overline{f}_{\vartheta_2}\leq 2\sum_{\theta,\vartheta}\int |f_{\theta}|^2|f_{\vartheta}|^2.
        \end{equation*}
        The result follows immediately.
    \end{proof}

	\subsection{Bounding the broad sets}\label{broad_sub}
	
	In this section, we fix a single function $f$ satisfying the hypotheses of Theorem \ref{squarefunction}. We apply the results of the previous subsection to pruned functions arising from $f$. 
 
    We will first need to define the decomposition. For each $I\in\cal{P}(\cal{O},R^{-1/2})$ and associated $\theta$, write
    \begin{equation*}
        \cal{G}_\theta=\left\{U\in\cal{U}_\theta:\fint_U|f_\theta|^2\geq \frac{\alpha^2}{2e^2N^2(\#\theta)^2}\right\},
    \end{equation*}
    and
    \begin{equation*}
        f_{N,\theta}=f_\theta\sum_{U\in\cal{G}_\theta}1_U,\quad f_N=\sum_\theta f_{N,\theta}.
    \end{equation*}
    Inductively, set
    \begin{equation*}
        \cal{G}_{\tau_k}:=\left\{U\in\cal{U}_{\tau_k}:\fint_U\sum_{\theta\prec\tau_k}|f_{k+1,\theta}|^2\geq \frac{\alpha^2}{2e^2N^2(\#\tau_k)^2}\right\},
    \end{equation*}
    and
    \begin{equation*}
        f_{k,\theta}:=f_{k+1,\theta}\sum_{U\in\cal{G}_{\tau_k}}1_{U},\quad f_k=\sum_{\theta}f_{k+1,\theta}.
    \end{equation*}
    We will also write
    \begin{equation*}
        f_{k,\theta}^{\cal{B}}=f_{k,\theta}-f_{k-1,\theta},\quad f_k^\cal{B}=\sum_\theta f_{k,\theta}^\cal{B}
    \end{equation*}

    We will write $g_{k,\tau}=\sum_{\theta\prec\tau}|f_{k,\theta}|^2$, $g_{k,\tau}^{\cal{B}}=\sum_{\theta\prec\tau}|f_{k,\theta}^{\cal{B}}|^2$.

    \begin{remark}
        Each $f_{k,\theta},f_{k,\theta}^\cal{B}$ is Fourier-supported in $\theta$.
    \end{remark}

    \begin{remark}[Pruning inequalities]
        For each $0\leq k<N$ and each $\theta$, we have the pointwise inequalities
        \begin{equation}\label{ineq:pruning}
            \max\big(|f_{k,\theta}|,|f_{k,\theta}^{\cal B}|\big)\leq|f_{k+1,\theta}|\leq |f_\theta|.
        \end{equation}
    \end{remark}
 
    An immediate consequence of the pruning is the following lemma, which demonstrates that the terms on the right-hand side of Prop. \ref{lasthighest} are appropriate for Theorem \ref{squarefunction} when applied to the $f_m^\cal{B}$.

	\begin{lemma}[Wave envelope expansion]\label{waveexp}
		\begin{enumerate}
			\item[(a)] For each $k$ and $\tau_k$, we have
			\begin{equation*}
				\int\Big|\sum_{\theta\prec\tau_k}|f_{m,\theta}^\mathcal{B}|^2*\chi_{\leq R_k/R}^\vee\Big|^2\leq \int\Big|\cal{A}_{\cal{U}_{\tau_k}}[g_{m,\tau_k}^{\cal{B}}]\Big|^2.
			\end{equation*}
			\item[(b)] If $N>k\geq m$, then
			\begin{equation*}
				\int\Big|\cal{A}_{\cal{U}_{\tau_k}}[g_{m,\tau_k}^{\cal{B}}]\Big|^2\leq\sum_{U\in\mathcal{G}_{\tau_k}}\mu(U)\left(\fint_U\sum_{\theta\prec\tau_k}|f_\theta|^2\right)^2.
			\end{equation*}
		\end{enumerate}
		
	\end{lemma}
	\begin{proof}
		(a): The Fourier support of $\sum_{\theta\prec\tau_k}|f_{m,\theta}^\mathcal{B}|^2*\chi_{\leq R_k/R}^\vee$ is contained in the set
		\begin{equation*}
			\bigcup_{\theta\prec\tau_k}(\theta-\theta)\cap B_{R_k/R}\subseteq U_{\tau_k,R}^*,
		\end{equation*}
		where $U_{\tau_k,R}^*$ is a rectangle of dimensions $R_k/R\times R^{-1}$ with long edge parallel to $\mathbf{t}_{c_{\tau_k}}$. Thus
		\begin{equation*}
			\begin{split}
				\int\Big|\sum_{\theta\prec\tau_k}|f_{m,\theta}^\mathcal{B}|^2*\chi_{\leq R_k/R}^\vee\Big|^2&=\int\Big|\sum_{\theta\prec\tau_k}\widehat{|f_{m,\theta}^\mathcal{B}|^2}\chi_{\leq R_k/R}\Big|^2\\
				&\leq\int\Big|\sum_{\theta\prec\tau_k}\widehat{|f_{m,\theta}^\mathcal{B}|^2}1_{U_{\tau_k,R}^*}\Big|^2\\
				&=\int\Big|\cal{A}_{\cal{U}_{\tau_k}}\Big[\sum_{\theta\prec\tau_k}|f_{m,\theta}^\mathcal{B}|^2\Big]\Big|^2
			\end{split}
		\end{equation*}
		as claimed.
		
		(b): Since $k\geq m$, $|f_{m,\theta}^\mathcal{B}|\leq|f_{k,\theta}|\leq|f_{k+1,\theta}|\leq|f_\theta|$ by the pruning, so
		\begin{equation*}
			\int\Big|\cal{A}_{\cal{U}_{\tau_k}}\Big[\sum_{\theta\prec\tau_k}|f_{m,\theta}^\mathcal{B}|^2\Big]\Big|^2\leq\int\left|\cal{A}_{\cal{U}_{\tau_k}}[g_{k,\tau_k}]\right|^2.
		\end{equation*}
  
        By the definition of the pruning, $g_{k,\tau_k}$ is supported on the union over $\cal{G}_{\tau_k}$, so
        \begin{equation*}
            \int\left|\cal{A}_{\cal{U}_{\tau_k}}[g_{k,\tau_k}]\right|^2=\sum_{U\in\cal{G}_{\tau_k}}\int_U\left|\cal{A}_{\cal{U}_{\tau_k}}[g_{k,\tau_k}]\right|^2.
        \end{equation*}
        On the other hand, $\cal{A}_{\cal{U}_{\tau_k}}[g_{k,\tau_k}]$ is constant on each such box, so that
        \begin{equation*}
            \sum_{U\in\cal{G}_{\tau_k}}\int_U\left|\cal{A}_{\cal{U}_{\tau_k}}[g_{k,\tau_k}]\right|^2=\sum_{U\in\cal{G}_{\tau_k}}\mu(U)\left(\fint_U\sum_{\theta\prec\tau_k}|f_{k,\theta}|^2\right)^2.
        \end{equation*}
	\end{proof}

 We now initiate a decomposition of the left-hand side of Theorem \ref{squarefunction}, which has the effect of isolating the contributions of the various $f_m^\cal{B}$. This portion of the argument follows closely the approach of \cite{GM1}, Section 3. Recall that $U_\alpha$ is defined as the set
	\begin{equation*}
		U_\alpha=\big\{x\in B_R:|f(x)|>\alpha\big\}.
	\end{equation*}
	We consider also the auxiliary set
	\begin{equation*}
		V_\alpha=\Big\{x\in B_R:|f_N(x)|>\Big(1-\frac{1}{2^{1/2}eN}\Big)\alpha\Big\}.
	\end{equation*}

	We note the relation
	\begin{equation}\label{replacement_containment}
		U_\alpha\subseteq V_\alpha;
	\end{equation}
    indeed, for any $x$, from the difference
    \begin{equation*}
			|f(x)-f_N(x)|\leq\sum_\theta\sum_{U\not\in\mathcal{G}_\theta}1_U(x)|f_\theta(x)|,
	\end{equation*}
    together with the fact that each $|f_\theta|$ is constant on the sets $U$, we see that for any $U\not\in\cal{G}_\theta$ and $x\in U$, we have the bound
    \begin{equation*}
        |f_\theta(x)|=\left(\fint_U|f_\theta|^2\right)^{1/2}\leq \frac{\alpha}{2^{1/2}eN(\#\theta)}.
    \end{equation*}
    Summing over the $\theta$, we obtain \eqref{replacement_containment}.
    
	By the definition of the prunings,
	\begin{equation}\label{pruningchoice}
		\begin{split}
			V_\alpha\subseteq\Big\{x\in V_\alpha:|f_0(x)|\geq (N^{4}+1)^{-1}|&f_N(x)|\Big\}\cup\bigcup_{m=1}^N\Big\{x\in V_\alpha:|f_m^{\mathcal{B}}|(x)\geq N^{-1}(1+N^{-4})^{-1}|f_N(x)|\Big\}\\
			&=:U_\alpha^0\cup\bigcup_{m=1}^NU_\alpha^m.
		\end{split}
	\end{equation}

    We may immediately verify that the set $U_\alpha^0$ is acceptable from the point of view of Theorem \ref{squarefunction}.
	
	\begin{proposition}[Case $m=0$]\label{broadc1} If $N>2$ then
		\[\alpha^4\mu(U_\alpha^0)\leq 20N^{10}\sum_{U\in\mathcal{G}_{\tau_0}}\mu(U)\left(\fint_U\sum_{\theta}|f_\theta|^2\right)^2.\]
	\end{proposition}
	\begin{proof}
		Take any $x\in U_\alpha^0$. Then $\cal{G}_{\tau_0}\neq\emptyset$, in particular is composed of the set $B_R$. Hence we have
		\begin{equation*}
            |f_0(x)|=|\sum_{\theta}1_{B_R}(x)f_{1,\theta}(x)|.
		\end{equation*}
        Thus, we have an inequality
        \begin{equation*}
            \int_{U_\alpha^0}|f_0|^2\leq \sum_{\theta}\int_{B_R}|f_{\theta}|^2.
        \end{equation*}
        Since $B_R\in\cal{G}_{\tau_0}$ we have
		\begin{equation*}
			\fint_{B_R}\sum_{\theta}|f_{\theta}|^2\geq\frac{\alpha^2}{2e^2N^2}.
		\end{equation*}
        Multiplying the two inequalities together,
        \begin{equation*}
            \frac{\alpha^2}{2e^2N^2}\int_{U_\alpha^0}|f_0|^2\leq \mu(U)\left(\fint_{B_R}\sum_{\theta}|f_\theta|^2\right)^2.
        \end{equation*}
        Finally, on $U_\alpha^0$ we have $|f_0(x)|\geq(1-\frac{1}{2^{1/2}eN})\frac{\alpha}{N^{4}+1}$; the result follows.
	\end{proof}

    In order to bound the contributions of the $U_\alpha^m$, we will need to assert an extra ``broad'' hypothesis. It will happen that this entails a high-frequency dominance property, which will allow us to apply Prop. \ref{lasthighest}. The basic result is in the following lemma.

    \begin{lemma}[Weak high-domination of bad parts]\label{wkhighdom} Let $1\leq m\leq\ell\leq N$ and $0\leq k<m$. Let $\tau_k$ be arbitrary.
		
		\begin{enumerate}
			\item[(a)] For each $x$, we have
            \begin{equation*}
                \left|\sum_{\tau_\ell\prec\tau_k}|f_{m,\tau_\ell}^\cal{B}|^2*\chi_{\leq R_\ell/R}^\vee(x)\right|\leq \frac{\alpha^2(\#\tau_\ell\prec\tau_k)}{2e^2N^2(\#\tau_\ell)^2}.
            \end{equation*}
   
			\item[(b)] 
            If $\frac{C\alpha}{2^{1/2}eN}\leq|f_{m,\tau_k}^\cal{B}(x)|$, then
            \begin{equation*}
                \sum_{\tau_\ell\prec\tau_k}|f_{m,\tau_\ell}^\cal{B}|^2(x)\leq \frac{C^2}{C^2-1}\left|\sum_{\tau_\ell\prec\tau_k}|f_{m,\tau_\ell}^\cal{B}|^2*\chi_{> R_\ell/R}^\vee(x)\right|.
            \end{equation*}
			
		\end{enumerate}
		
	\end{lemma}
	\begin{proof}
		(a):  By the low lemma \ref{lowlemma},
        \begin{equation*}
            \sum_{\tau_\ell\prec\tau_k}|f_{m,\tau_\ell}^\cal{B}|^2*\chi_{\leq R_\ell/R}^\vee=\sum_{\tau_\ell\prec\tau_k}\sum_{\theta\prec\tau_\ell}|f_{m,\theta}^\cal{B}|^2*\chi_{\leq R_\ell/R}^\vee.
        \end{equation*}
        If $x\in U\in\cal{U}_{\tau_\ell}\setminus\cal{G}_{\tau_\ell}$, then a straightforward calculation supplies
        \begin{equation*}
            \sum_{\theta\prec\tau_\ell}|f_{m,\theta}^\cal{B}|^2*\chi_{\leq R_\ell/R}^\vee(x)=\fint_U\sum_{\theta\prec\tau_\ell}|f_{m,\theta}^\cal{B}|^2*\chi_{\leq R_\ell/R}^\vee.
        \end{equation*}
        By the pruning inequalities \eqref{ineq:pruning},
        \begin{equation*}
            \fint_U\sum_{\theta\prec\tau_\ell}|f_{m,\theta}^\cal{B}|^2*\chi_{\leq R_\ell/R}^\vee\leq\fint_U\sum_{\theta\prec\tau_\ell}|f_{\ell,\theta}^\cal{B}|^2*\chi_{\leq R_\ell/R}^\vee.
        \end{equation*}
        By the definition of $\cal{G}_{\tau_\ell}$,
        \begin{equation*}
            \fint_U\sum_{\theta\prec\tau_\ell}|f_{\ell,\theta}^\cal{B}|^2*\chi_{\leq R_\ell/R}^\vee\leq \frac{\alpha^2}{2e^2N^2(\#\tau_\ell)^2}.
        \end{equation*}
        Thus, summing over $\tau_\ell\prec\tau_k$, we arrive at the estimate
        \begin{equation*}
            \sum_{\tau_\ell\prec\tau_k}|f_{m,\tau_\ell}^\cal{B}|^2*\chi_{\leq R_\ell/R}^\vee\leq \frac{\alpha^2(\#\tau_\ell\prec\tau_k)}{2e^2N^2(\#\tau_\ell)^2}.
        \end{equation*}
		
		(b): The assumption and Cauchy--Schwarz give
        \begin{equation*}
            \frac{C^2\alpha^2}{2e^2N^2}\leq(\#\tau_\ell\prec\tau_k)\sum_{\tau_\ell\prec\tau_k}|f_{m,\tau_\ell}^\cal{B}|^2.
        \end{equation*}
        Assuming to the contrary that
        \begin{equation*}
            \sum_{\tau_\ell\prec\tau_k}|f_{m,\tau_\ell}^\cal{B}|^2(x)<C^2\left|\sum_{\tau_\ell\prec\tau_k}|f_{m,\tau_\ell}^\cal{B}|^2*\chi_{\leq R_\ell/R}^\vee(x)\right|,
        \end{equation*}
        we conclude from (a) that
        \begin{equation*}
            \frac{C^2\alpha^2}{2e^2N^2}<\frac{C^2\alpha^2N^2(\#\tau_\ell\prec\tau_k)^2}{2e^2(\#\tau_\ell)^2},
        \end{equation*}
        a contradiction. Thus we must have
        \begin{equation*}
            \sum_{\tau_\ell\prec\tau_k}|f_{m,\tau_\ell}^\cal{B}|^2(x)\leq \frac{C^2}{C^2-1}\left|\sum_{\tau_\ell\prec\tau_k}|f_{m,\tau_\ell}^\cal{B}|^2*\chi_{> R_\ell/R}^\vee(x)\right|.
        \end{equation*}
	\end{proof}    

    Finally, we use the above estimates to control the integrals of the $f_m^\cal{B}$ on broad sets by high-frequency integrals of the square functions. Define the $m$th $(1\leq m\leq N)$ broad sets in $U_\alpha$ to be as follows. Fix any $\tau_k,\tau_k'\prec\tau_{k-1}$ distinct, and write $\ell=\max(m-1,k)$. We define
	\begin{equation}\label{broadsetdef}
		\mathrm{Br}_\alpha^m(\tau_k,\tau_k')=\Big\{x\in B_R:(1-2^{-1/2}e^{-1}N^{-1})e^{-1}N^{-1}\alpha\leq |f_{m,\tau_{k-1}}^\cal{B}(x)|\leq \bf{p}N|f_{m,\tau_k}^{\cal{B}}(x)f_{m,\tau_k'}^\cal{B}(x)|^{1/2}\Big\}.
	\end{equation}
	
	\begin{proposition}[High domination of broad integrals]\label{broadc2} Let $N>2$. Let $1\leq k\leq m\leq N$. Suppose $\tau_k,\tau_k'\prec\tau_{k-1}$ are distinct. Write $\ell=\max(m-1,k)$. Then we have
        \begin{equation*}
            \int_{\mathrm{Br}_\alpha^m(\tau_k,\tau_k')}|f_{m,\tau_k}^\cal{B}f_{m,\tau_k'}^\cal{B}|^2\leq 6\bf q^{-2}\bf{p}^2\int_{\Q_p^2}\Big|\sum_{\tau_{\ell}\prec\tau_{k-1}}|f_{m,\tau_{\ell}}^{\cal{B}}|^2*\chi_{>R_{\ell}/R}^\vee\Big|^2.
        \end{equation*}
	\end{proposition}
	\begin{proof}

        By local bilinear restriction \ref{locbilres}, using the fact that $\bf{p}\geq|2|^{-1}$, in either case $k\leq m-1$ or $k\geq m$, we achieve the bound
        \begin{equation*}
            \int_{\operatorname{Br}_\alpha^m(\tau_k,\tau_k')}|f_{m,\tau_k}^\mathcal{B}f_{m,\tau_k'}^\mathcal{B}|^2\leq \bf q^{-2}\bf{p}^2\int_{\operatorname{Br}_\alpha^m(\tau_k,\tau_k')+B(0,R_{\ell})}\Big|\sum_{\tau_{\ell}\prec\tau_{k-1}}|f_{m,\tau_{\ell}}^{\cal{B}}|^2\Big|^2.
        \end{equation*}
        
        For each $x\in\mathrm{Br}_\alpha^m(\tau_k,\tau_k')$, we have that $|f_{m,\tau_{k-1}}^\cal{B}(x)|\geq 0.915\cdot e^{-1}N^{-1}\alpha$. Thus, by the weak high-domination lemma \ref{wkhighdom},
        \begin{equation*}
            \sum_{\tau_{\ell}\prec\tau_{k-1}}|f_{m,\tau_{\ell}}^{\mathcal{B}}|^2(x)\leq 2.5\Big|\sum_{\tau_{\ell}\prec\tau_{k-1}}|f_{m,\tau_{\ell}}^{\mathcal{B}}|^2*\chi_{> R_{\ell}/R}^\vee(x)\Big|.
        \end{equation*}
        On the other hand, both sides are constant at scale $R_{\ell}$, so
        \begin{equation*}
            \begin{split}
            \int_{\operatorname{Br}_\alpha^m(\tau_k,\tau_k')+B(0,R_{\ell})}&\Big|\sum_{\tau_{\ell}\prec\tau_{k-1}}|f_{m,\tau_{\ell}}^{\cal{B}}|^2\Big|^2\\
            &\leq 6\int_{\operatorname{Br}_\alpha^m(\tau_k,\tau_k')+B(0,R_{\ell})}\Big|\sum_{\tau_{\ell}\prec\tau_{k-1}}|f_{m,\tau_{\ell}}^{\cal{B}}|^2*\chi_{>R_{\ell}/R}^\vee\Big|^2.
            \end{split}
        \end{equation*}
		
	\end{proof}

    \begin{remark}
        In the above proposition, we see the appearance of the particular constant used in the choice of pruning. It follows that, if the assumptions on the broad sets are weakened, one may correspondingly strengthen the assumption on the good envelopes $\cal{G}_{\tau_k}$, and improve the strengths of Theorems \ref{squarefunction} and \ref{smallcap}.
    \end{remark}
	
	\subsection{Broad/narrow analysis}\label{bn_sub}
	
	Combining Prop.'s \ref{broadc1} and \ref{broadc2}, we produced the desired bounds on the subset of the superlevel set for which $f$ is sufficiently broad at some scale. In this subsection, we perform a broad/narrow analysis to produced the desired wave envelope estimate in each cube of sidelength $R$.

    More precisely, this section is dedicated to establishing the following proposition.
	\begin{proposition}[Local wave envelope estimate]\label{broadnarrow} Suppose $N>6$. For each cube $B_R$ of sidelength $R$ and each $\alpha>0$,
		\begin{equation*}
			\alpha^4\mu\Big(\{x\in B_R:|f(x)|>\alpha\}\Big)\leq 20\bf q^{-3}\bf{p}^{10}N^{10}\sum_{0\leq s\leq N}\sum_{\tau_s}\sum_{U\in\mathcal{G}_{\tau_s}}\mu(U)\left(\fint_U\sum_{\theta\prec\tau_s}|f_{\theta}|^2\right)^2.
		\end{equation*}
		
	\end{proposition}
    
	\begin{lemma}[Narrow lemma]\label{narrow} Suppose $1\leq k\leq N$ and $\tau_{k-1}$ is arbitrary. Then, for each $x$, either
		\begin{equation}\label{broadest}
			|f_{m,\tau_{k-1}}^\mathcal{B}(x)|\leq \bf{p}N\max_{\substack{\tau_{k}\neq\tau_{k}'\\\tau_{k},\tau_{k}'\prec\tau_{k-1}}}|f_{m,\tau_{k}}^\mathcal{B}(x)f_{m,\tau_{k}'}^\mathcal{B}(x)|^{1/2}
		\end{equation}
		or
		\begin{equation}\label{narrowest}
			|f_{m,\tau_{k-1}}^\mathcal{B}(x)|\leq\left(1+\frac{1}{N-1}\right)\max_{\tau_{k}\prec\tau_{k-1}}|f_{m,\tau_{k}}^\mathcal{B}(x)|.
		\end{equation}
		
	\end{lemma}
	\begin{proof}
		Fix $\tau_{k}'\prec\tau_{k-1}$ which realizes the maximum
		\begin{equation*}
			|f_{m,\tau_{k}'}^\mathcal{B}(x)|=\max_{\tau_{k}\prec\tau_{k-1}}|f_{m,\tau_{k}}^\mathcal{B}(x)|.
		\end{equation*}
		Suppose \eqref{narrowest} fails. Then, since $f_{m,\tau_{k-1}}^\mathcal{B}(x)=\sum_{\tau_{k}\prec\tau_{k-1}}f_{m,\tau_{k}}^\mathcal{B}(x)$, we have the inequality
		\begin{equation*}
			|f_{m,\tau_{k-1}}^\mathcal{B}(x)-\sum_{\tau_{k}\neq\tau_{k}'}f_{m,\tau_{k}}^\mathcal{B}(x)|<\left(1+\frac{1}{N-1}\right)^{-1}|f_{m,\tau_{k-1}}^\mathcal{B}(x)|.
		\end{equation*}
		On the other hand,
		\begin{equation*}
			|f_{m,\tau_{k-1}}^\mathcal{B}(x)-\sum_{\tau_{k}\neq\tau_{k}'}f_{m,\tau_{k}}^\mathcal{B}(x)|\geq|f_{m,\tau_{k-1}}^\mathcal{B}(x)|-(\#\tau_{k}\prec\tau_{k-1})\max_{\tau_{k}\neq\tau_{k}'}|f_{m,\tau_{k}}^\mathcal{B}(x)|;
		\end{equation*}
		thus
		\begin{equation*}
			(\#\tau_{k}\prec\tau_{k-1})\max_{\tau_{k}\neq\tau_{k}'}|f_{m,\tau_{k}}^\mathcal{B}(x)|>\left(1-\left(1+\frac{1}{N-1}\right)^{-1}\right)|f_{m,\tau_{k-1}}^\mathcal{B}(x)|.
		\end{equation*}
		Relating the above to \eqref{broadest}, for each $\tau_{k}\prec\tau_{k-1}$,
		\begin{equation*}
			|f_{m,\tau_{k}}^\mathcal{B}(x)|\leq|f_{m,\tau_{k}}^\mathcal{B}(x)f_{m,\tau_{k}'}^\mathcal{B}(x)|^{1/2},
		\end{equation*}
		and thus
		\begin{equation*}
			|f_{m,\tau_{k-1}}^\mathcal{B}(x)|<(\#\tau_{k})\left(1-\left(1+\frac{1}{N-1}\right)^{-1}\right)^{-1}\max_{\tau_{k}\neq\tau_{k}'}|f_{m,\tau_{k}}^\mathcal{B}(x)f_{m,\tau_{k}'}^\mathcal{B}(x)|^{1/2}.
		\end{equation*}
		The conclusion follows from the identities
		\begin{equation*}
			\left(1-\left(1+\frac{1}{N-1}\right)^{-1}\right)^{-1}=N
		\end{equation*}
		and
		\begin{equation*}
			(\#\tau_{k}\prec\tau_{k-1})\leq\bf{p}.
		\end{equation*}
		
	\end{proof}
	
	We wish to use this to divide the integral of $|f_m^\mathcal{B}|^4$ into broad and narrow parts, with a small constant on narrow parts. For the narrow component, we wish to relate $\int |f_m^\mathcal{B}|^4$ to $\sum_\tau\int|f_{m,\tau}^\mathcal{B}|^4$, so that we may further decompose each $f_{m,\tau}^\mathcal{B}$ into broad and narrow components and proceed inductively.
	
	\begin{definition} We define $\text{Broad}_{1,m}$ to be the set
		\begin{equation*}
			\text{Broad}_{1,m}=\left\{x\in U_\alpha^m:\,|f_m^\mathcal{B}(x)|\leq \bf{p}N\max_{\tau_1\neq \tau_1'}|f_{m,\tau_1}^\mathcal{B}(x)f_{m,\tau_1'}^\mathcal{B}(x)|^{1/2}\right\}.
		\end{equation*}
		The complementary set $\text{Narrow}_{1,m}$ is defined as $U_\alpha^m\setminus\text{Broad}_{1,m}$.
	\end{definition}
	
	\begin{lemma}[Decoupling the narrow part ($k=1$)]\label{broaddec1} It holds that
		\begin{equation*}
			\int_{\mathrm{Narrow}_{1,m}}|f_{m}^\mathcal{B}|^4\leq\left(1+\frac{1}{N-1}\right)^4\sum_{\tau_1}\int_{\mathrm{Narrow}_{1,m}}|f_{m,\tau_1}^\mathcal{B}|^4.
		\end{equation*}
	\end{lemma}
	\begin{proof}
		
		If $x\in\text{Narrow}_{1,m}$, then by the narrow lemma \ref{narrow}
		\begin{equation*}
			|f_m^\mathcal{B}(x)|\leq\left(1+\frac{1}{N-1}\right)|f_{m,\tau_1}^\mathcal{B}(x)|
		\end{equation*}
		for a suitable $\tau_1$. Thus,
		\begin{equation*}
			|f_m^\mathcal{B}(x)|\leq\left(1+\frac{1}{N-1}\right)\left(\sum_{\tau_1}|f_{m,\tau_1}^\mathcal{B}(x)|^4\right)^{1/4},
		\end{equation*}
		and hence
		\begin{equation*}
			\int_{\mathrm{Narrow}_{1,m}}|f_{m}^\mathcal{B}|^4\leq\left(1+\frac{1}{N-1}\right)^4\sum_{\tau_1}\int_{\mathrm{Narrow}_{1,m}}|f_{m,\tau_1}^\mathcal{B}|^4.
		\end{equation*}
		
	\end{proof}
	
	\begin{definition}
		Write, for each $\tau_1$,
		\begin{equation*}
			\text{Broad}_{2,m}(\tau_1):=\left\{x\in\text{Narrow}_{1,m}:|f_{m,\tau_1}^\mathcal{B}(x)|\leq\bf{p}N\max_{\substack{\tau_2\neq\tau_2'\\\tau_2,\tau_2'\prec\tau_1}}|f_{m,\tau_2}^\mathcal{B}(x)f_{m,\tau_2'}^\mathcal{B}(x)|^{1/2}\right\}.
		\end{equation*}
		Write also $\text{Narrow}_{2,m}(\tau_1):=\text{Narrow}_{1,m}\setminus\text{Broad}_{2,m}(\tau_1)$.
	\end{definition}
	
	\begin{definition}
		Let $2\leq k<N$. Suppose $\tau_k\prec\tau_{k-1}$. We inductively write
		\begin{equation*}
			\begin{split}
				\text{Broad}_{k+1,m}(\tau_k):=\left\{x\in\text{Narrow}_{k,m}(\tau_{k-1}):|f_{m,\tau_{k}}^\mathcal{B}(x)|\leq \bf{p}N\max_{\substack{\tau_{k+1}\neq\tau_{k+1}'\\\tau_{k+1},\tau_{k+1}'\prec\tau_{k}}}|f_{m,\tau_{k+1}}^\mathcal{B}(x)f_{m,\tau_{k+1}'}^\mathcal{B}(x)|^{1/2}\right\}
			\end{split}
		\end{equation*}
		and $\text{Narrow}_{k+1,m}(\tau_k):=\text{Narrow}_{k,m}(\tau_{k-1})\setminus\text{Broad}_{k+1,m}(\tau_k)$.
	\end{definition}
	
	\begin{lemma}[Decoupling the narrow part ($k\geq 2$)]\label{broaddec2} Fix any $2\leq k\leq N$. Then, for each $\tau_{k-1}$,
		\begin{equation*}
			\int_{\mathrm{Narrow}_{k,m}(\tau_{k-1})}|f_{m,\tau_{k-1}}^\mathcal{B}|^4\leq\left(1+\frac{1}{N-1}\right)^4\sum_{\tau_k\prec\tau_{k-1}}\int_{\mathrm{Narrow}_{k,m}(\tau_{k-1})}|f_{m,\tau_k}^\mathcal{B}|^4.
		\end{equation*}
		
	\end{lemma}
	\begin{proof}
		The argument is identical to \ref{broaddec1}.
	\end{proof}
	
	Combining Lemmas \ref{broaddec1} and \ref{broaddec2}, we conclude
	\begin{equation*}
		\begin{split}
			\int_{U_\alpha^m}|f_m^\mathcal{B}|^4&\leq\left(1+\frac{1}{N-1}\right)^{4N}\sum_{\tau_{N-1}}\int_{\mathrm{Narrow}_{N,m}(\tau_{N-1})}\sum_{\tau_N\prec\tau_{N-1}}|f_{m,\tau_N}^\mathcal{B}|^4\\
			&+\sum_{k=1}^{N}\left(1+\frac{1}{N-1}\right)^{4(k-1)}\sum_{\tau_{k-1}}\int_{\mathrm{Broad}_{k,m}(\tau_{k-1})}|f_{m,\tau_{k-1}}^\mathcal{B}|^4.
		\end{split}
	\end{equation*}
	
	Our next steps are bounding each of the summands in turn.
	
	\begin{lemma}[Narrow bound]\label{narrowbound} We have
		\begin{equation*}
			\sum_{\tau_{N-1}}\int_{B_R}\sum_{\theta\prec\tau_{N-1}}|f_{m,\theta}^\mathcal{B}|^4\leq\sum_\theta\sum_{U\in\mathcal{G}_\theta}\mu(U)\left(\fint_U|f_\theta|^2\right)^2.
		\end{equation*}

	\end{lemma}
	\begin{proof}
        By the pruning inequalities \eqref{ineq:pruning}, for each $\theta$ we may take $|f_{m,\theta}^{\cal{B}}|\leq|f_{N,\theta}|$. By the definition of the pruning, for each $\theta$,
		\begin{equation*}
			\int_{B_R}|f_{N,\theta}|^4=\int_{B_R}\Big|\sum_{U\in\mathcal{G}_\theta}1_Uf_\theta\Big|^4\leq\sum_{U\in\mathcal{G}_\theta}\int_{U}|f_\theta|^4.
		\end{equation*}
        Each $|f_\theta|$ is constant on the blocks $U$, so for each $\theta$ and $U\in\cal{G}_\theta$ we have
        \begin{equation*}
            \int_U|f_\theta|^4=\mu(U)\left(\fint_U|f_\theta|^2\right)^2.
        \end{equation*}
	\end{proof}

	\begin{lemma}[Broad bound]\label{broadbound2} We have, for each $1\leq k\leq N$,
		\begin{equation*}
            \sum_{\tau_{k-1}}\int_{\mathrm{Broad}_{k,m}(\tau_{k-1})}|f_{m,\tau_{k-1}}^\mathcal{B}|^4\leq 18\bf q^{-3}\bf{p}^{10}N^{5}\sum_{m\leq s\leq N}\sum_{\tau_s}\sum_{U\in\mathcal{G}_{\tau_s}}\mu(U)\left(\fint_U\sum_{\theta\prec\tau_s}|f_\theta|^2\right)^2
        \end{equation*}
	\end{lemma}
	\begin{proof}
		By the definition of the broad set, for each $\tau_{k-1}$ and each $x\in\mathrm{Broad}_{k,m}(\tau_{k-1})$ there is some pair $\tau_k\neq\tau_k'\prec\tau_{k-1}$ such that $|f_{m,\tau_{k-1}}^\mathcal{B}(x)|\leq \bf{p}N|f_{m,\tau_k}^\mathcal{B}(x)f_{m,\tau_k'}^\mathcal{B}(x)|^{1/2}$, as well as 
  $$N^{-1}(1-(2^{1/2}eN)^{-1})(1+(N-1)^{-1})^{-(k-1)}\alpha\leq|f_{m,\tau_{k-1}}^\cal{B}(x)|,$$ i.e.
		\begin{equation*}
			\mathrm{Broad}_{k,m}(\tau_{k-1})\subseteq\bigcup_{\substack{\tau_k,\tau_k'\prec\tau_{k-1}\\\tau_k\neq\tau_k'}}\mathrm{Br}_\alpha^m(\tau_k,\tau_k').
		\end{equation*}
		Thus,
		\begin{equation*}
			\sum_{\tau_{k-1}}\int_{\mathrm{Broad}_{k,m}(\tau_{k-1})}|f_{m,\tau_{k-1}}^\mathcal{B}|^4\leq \bf{p}^4N^{4}\sum_{\tau_{k-1}}\sum_{\substack{\tau_k,\tau_k'\prec\tau_{k-1}\\\tau_k\neq\tau_k'}}\int_{\mathrm{Br}_\alpha^m(\tau_{k},\tau_k')}|f_{m,\tau_{k}}^\mathcal{B}f_{m,\tau_k'}^\cal{B}|^2.
		\end{equation*}
        By Prop. \ref{broadc2}, for each $\tau_{k-1}$ and each distinct $\tau_k,\tau_k'\prec\tau_{k-1}$, then
        \begin{equation*}
            \int_{\mathrm{Br}_\alpha^m(\tau_{k},\tau_k')}|f_{m,\tau_{k}}^\mathcal{B}f_{m,\tau_k'}^\cal{B}|^2\leq 6\bf q^{-2}\bf{p}^2\int_{\Q_p^2}\left|\sum_{\tau_{\ell}\prec\tau_{k-1}}|f_{m,\tau_{\ell}}^{\cal{B}}|^2*\chi_{>R_{\ell}/R}^\vee\right|^2,
        \end{equation*}
        where $\ell=\max(m-1,k)$. By Prop. \ref{lasthighest} and the wave envelope expansion lemma,
        \begin{equation*}
            \int_{\Q_p^2}\left|\sum_{\tau_{\ell}\prec\tau_{k-1}}|f_{m,\tau_{\ell}}^\mathcal{B}|^2*\chi_{> R_{\ell}/R}^\vee\right|^2\leq 3\bf q^{-1}\bf{p}^4N\sum_{s=\ell+1}^N\sum_{\tau_s\prec\tau_{k-1}}\sum_{U\in\mathcal{G}_{\tau_s}}\mu(U)\left(\fint_U\sum_{\theta\prec\tau_s}|f_\theta|^2\right)^2.
        \end{equation*}
        Combining the above, we obtain
        \begin{equation*}
            \sum_{\tau_{k-1}}\int_{\mathrm{Broad}_{k,m}(\tau_{k-1})}|f_{m,\tau_{k-1}}^\mathcal{B}|^4\leq 18\bf q^{-3}\bf{p}^{10}N^{5}\sum_{m\leq s\leq N}\sum_{\tau_s}\sum_{U\in\mathcal{G}_{\tau_s}}\mu(U)\left(\fint_U\sum_{\theta\prec\tau_s}|f_\theta|^2\right)^2.
        \end{equation*}
		
	\end{proof}
	
	\begin{proof}[Proof of Prop. \ref{broadnarrow}]
		
		By \eqref{replacement_containment},
		\begin{equation*}
			\alpha^4\mu(U_\alpha)\leq \alpha^4\mu(V_\alpha).
		\end{equation*}
		Write
		\begin{equation*}
			\alpha^4\mu(V_\alpha)\leq\sum_{m=0}^N\alpha^4\mu(U_\alpha^m)
		\end{equation*}
		where the sets in the right-hand side are as defined in \eqref{pruningchoice}. By Prop. \ref{broadc1},
        \begin{equation*}
            \alpha^4\mu(U_\alpha^0)\leq 20N^{10}\sum_{U\in\mathcal{G}_{\tau_0}}\mu(U)\left(\fint_U\sum_{\theta}|f_\theta|^2\right)^2.
        \end{equation*}
        We now fix some $1\leq m\leq N$. Over $U_\alpha^m$, we have a lower bound on $f_m^\cal{B}$ implying
        \begin{equation*}
            \alpha^4\mu(U_\alpha^m)\leq N^4(1+N^{-4})^{10}(1-(2^{1/2}eN)^{-1})^{-1}\int_{U_\alpha^m}|f_m^\cal{B}|^4.
        \end{equation*}
        By the definition of the broad/narrow sets above, we may bound
        \begin{equation*}
            \begin{split}
            \int_{U_\alpha^m}|f_m^\cal{B}|^4&\leq\left(1+\frac{1}{N}\right)^{4N}\sum_{\tau_{N-1}}\int_{\mathrm{Narrow}_{N,m}(\tau_{N-1})}\sum_{\theta\prec\tau_{N-1}}|f_{\theta,m}^\mathcal{B}|^4\\
            &+\sum_{k=1}^{N}\left(1+\frac{1}{N}\right)^{4(k-1)}\sum_{\tau_{k-1}}\int_{\mathrm{Broad}_{k,m}(\tau_{k-1})}|f_{\tau_{k-1},m}^\mathcal{B}|^4.
            \end{split}
        \end{equation*}
		By the narrow bound \ref{narrowbound},
		\begin{equation*}
			\sum_{\tau_{N-1}}\int_{B_R}\sum_{\theta\prec\tau_{N-1}}|f_{m,\theta}^\mathcal{B}|^4\leq\sum_\theta\sum_{U\in\mathcal{G}_\theta}\mu(U)\left(\fint_U|f_\theta|^2\right)^2.
		\end{equation*}
		By the broad bound \ref{broadbound2},
		\begin{equation*}
			\sum_{\tau_{k-1}}\int_{\mathrm{Broad}_{k,m}(\tau_{k-1})}|f_{m,\tau_{k-1}}^\mathcal{B}|^4\leq 18\bf q^{-3}\bf{p}^{10}N^{5}\sum_{m\leq s\leq N}\sum_{\tau_s}\sum_{U\in\mathcal{G}_\tau}\mu(U)\left(\fint_U\sum_{\theta\prec\tau_s}|f_{\theta}|^2\right)^2.
		\end{equation*}
		
		Thus, for each $1\leq m\leq N$,
		\begin{equation*}
			\alpha^4\mu(U_\alpha^m)\leq 18(1+N^{-4})^{10}(2^{1/2}eN)^{-1})^{-1}\bf{p}^{10}N^{9}\sum_{m\leq s\leq N}\sum_{\tau_s}\sum_{U\in\mathcal{G}_\tau}\mu(U)\left(\fint_U\sum_{\theta\prec\tau_s}|f_{\theta}|^2\right)^2,
		\end{equation*}
		and hence, using the $m=0$ estimate and trivial estimates using $N>6$, we conclude that
		\begin{equation}\label{superlevel_final}
			\alpha^4\mu(U_\alpha)\leq 20\bf q^{-3}\bf{p}^{10}N^{10}\sum_{0\leq s\leq N}\sum_{\tau_s}\sum_{U\in\mathcal{G}_{\tau_s}}\mu(U)\left(\fint_U\sum_{\theta\prec\tau_s}|f_{\theta}|^2\right)^2,
		\end{equation}
		as claimed.

  \end{proof}

	\subsection{Reduction to local estimates}\label{local_sub}
	
	In the above subsections we produced bounds on the measure of the set $U_\alpha=\{x\in B_R:|f(x)|>\alpha\}$. In this subsection we note that, if we can prove Theorem \ref{squarefunction} in the special case that $\{x\in\R^2:|f(x)|>\alpha\}\subseteq Q_R$ for a suitable cube $Q_R$ of radius $R$, then we can conclude that Theorem \ref{squarefunction} is true in the general case.
	
	\begin{proof}[Proof that Prop. \ref{broadnarrow} implies Theorem \ref{squarefunction}]

        Let $f$ be as in the hypothesis of Theorem \ref{squarefunction}. Then, for each metric ball $Q_R\subseteq\Q_p^2$ of radius $R$, the function $1_{Q_R}f$ satisfies the hypothesis of Prop. \ref{broadnarrow}. Thus, we have
        \begin{equation*}
            \mu(\{x\in\Q_p^2:|1_{Q_R}f|>\alpha\})\leq 20\bf q^{-3}\bf{p}^{10}N^{10}\sum_{R^{-1/2}\leq s\leq 1}\sum_{\tau_s}\sum_{U\in\cal{G}_{\tau_s}}\mu(U)\left(\fint_U\sum_{\theta\prec\tau_s}|(1_{Q_R}f)_\theta|^2\right)^2.
        \end{equation*}
        Summing over the $Q_R$, we obtain that
        \begin{equation*}
            \mu(\{x\in\Q_p^2:|f(x)|>\alpha\})\leq 20\bf q^{-3}\bf{p}^{10}N^{10}\sum_{R^{-1/2}\leq s\leq 1}\sum_{\tau_s}\sum_{U\in\cal{G}_{\tau_s}}\mu(U)\sum_{Q_R}\left(\fint_U\sum_{\theta\prec\tau_s}|(1_{Q_R}f)_\theta|^2\right)^2.
        \end{equation*}
        By Minkowski, for each $\tau_s$ and $U\in\cal{G}_{\tau_s}$,
        \begin{equation*}
            \sum_{Q_R}\left(\fint_U\sum_{\theta\prec\tau_s}|(1_{Q_R}f)_\theta|^2\right)^2\leq\left(\fint_U\sum_{\theta\prec\tau_s}\left(\sum_{Q_R}|(1_{Q_R}f)_\theta|^2\right)^2\right)^2.
        \end{equation*}
        By elementary properties of the ultrametric, for each $Q_R$ and each $\theta$ we have
        \begin{equation*}
            (1_{Q_R}f)_\theta=1_{Q_R}f_\theta,
        \end{equation*}
        so that in fact
        \begin{equation*}
            \left(\fint_U\sum_{\theta\prec\tau_s}\left(\sum_{Q_R}|(1_{Q_R}f)_\theta|^2\right)^2\right)^2=\left(\fint_U\sum_{\theta\prec\tau_s}|f_\theta|^2\right)^2.
        \end{equation*}
        The result follows immediately.
	\end{proof}
	
	\section{Proof of Theorem \ref{smallcap}}\label{smallcapproof}

    Theorem \ref{smallcap} will be proved via first establishing the critical case $(r,q)=(2+\frac{2}{\beta},\frac{2+2\beta^{-1}}{2\beta^{-1}-1})$, and then interpolating with the easier endpoints $(\infty,1)$ and $(3,\infty)$, together with H\"older inequalities. We will repeatedly cite the result of Theorem \ref{squarefunction}, and we will abbreviate the constants as
    \begin{equation*}
        \alpha^4|\{x:|f(x)|>\alpha\}|\leq CN^{E_2}\sum_{\substack{k\in\bf{p}^\Z\\R^{-1/2}\leq s\leq 1}}\sum_{\tau_k}\sum_{U\in\mathcal{G}_{\tau_k}}\mu(U)\left(\fint_U\sum_{\theta\prec\tau_k}|f_\theta|^2\right)^2.
    \end{equation*}
    We will in fact make use of the more refined quantities in \eqref{superlevel_final}.
 
	We begin with the partial decoupling statement, using the right-hand side of Theorem \ref{squarefunction}.
	
	\begin{proposition}\label{partial} Suppose $r\geq 4$, $\lambda>0$, and $\mathfrak{C}_r>1$. Let $0\leq k\leq N$ be arbitrary, and fix a canonical scale cap $\tau_k$. Suppose as before that $\Gamma_\beta(R^{-1})$ is a partition of $\mathcal{N}_{R^{-1}}(\mathbb{P}^1)$ into approximate $R^{-\beta}\times R^{-1}$ boxes $\gamma$. Assume $f=\sum_\gamma f_\gamma$ satisfies the following regularity properties.
		\begin{itemize}
			\item[(a)] $\|f_\gamma\|_\infty\leq \lambda$ for each $\gamma$.
			\item[(b)] $\|f_\gamma\|_2^2\leq\lambda^{2-r}\frak{C}_{r}\|f_\gamma\|_r^r$ for each $\gamma$ and each $r\geq 1$.
		\end{itemize}
		Write $\gamma_k$ for boxes of dimensions $\max(R^{-\beta},R_k/R)\times R^{-1}$. Then
		\begin{equation}\label{localenvelope}
			\begin{split}
				\sum_{U\in\mathcal{G}_{\tau_k}}\mu(U)\left(\fint_U\sum_{\theta\prec\tau_k}|f_\theta|^2\right)^2&\leq \frak{C}_{r}
				\big[2^{\frac{1}{2}}eN(\#\tau_k)\alpha^{-1}\big]^{r-4}\\
				&\times\left(\max_{\gamma_k\prec\tau_k}\#(\gamma\prec\gamma_k)\times\#(\gamma\prec\tau_k)\right)^{\frac{r}{2}-1}\sum_{\gamma\prec\tau_k}\|f_\gamma\|_r^r.
			\end{split}
		\end{equation}
		
	\end{proposition}
	\begin{proof}
		
		For each $\theta\prec\tau_k$, the small caps $\gamma_k\prec\theta$ are $\max(R^{-\beta},R_k/R)\geq R_k/R$-separated. Fix any $U\in\mathcal{G}_{\tau_k}$.  Since $U\|U_{\tau_k,R}$ has dimensions $R/R_k\times R$, we conclude that the $f_{\gamma_k}$ are locally orthogonal on $U$. Thus
		\begin{equation*}
			\int_U\sum_{\theta\prec\tau_k}|f_{\theta}|^2=\int_U\sum_{\gamma_k\prec\tau_k}|f_{\gamma_k}|^2,
		\end{equation*}
		and so, appealing to the definition of $\mathcal{G}_{\tau_k}$,
		\begin{equation*}
			\frac{\alpha^2}{2e^2N^2(\#\tau_k)^2}\leq\fint_U\sum_{\gamma_k\prec\tau_k}|f_{\gamma_k}|^2.
		\end{equation*}
		Multiplying the left-hand side of \eqref{localenvelope} by the $(\frac{r}{2}-2)$-power of the latter display, we obtain the estimate
		\begin{equation}\label{envest}
			\sum_{U\in\mathcal{G}_{\tau_k}}\mu(U)\left(\fint_U\sum_{\theta\prec\tau_k}|f_\theta|^2\right)^2\leq \big[2^{1/2}eN(\#\tau_k)\alpha^{-1}\big]^{r-4}\sum_{U\in\mathcal{G}_{\tau_k}}\mu(U)\left(\fint_U\sum_{\gamma_k\prec\tau_k}|f_{\gamma_k}|^2\right)^{\frac{r}{2}}.
		\end{equation}
		Uniformity assumption (a) implies
		\begin{equation*}
			\Big\|\sum_{\gamma_k\prec\tau_k}|f_{\gamma_k}|^2\Big\|_\infty\leq\lambda^2\big[\max_{\gamma_k\prec\tau_k}\#(\gamma\prec\gamma_k)\big]\times\#(\gamma\prec\tau_k).
		\end{equation*}
		By removing factors of $\big\|\sum_{\gamma_k\prec\tau_k}|f_{\gamma_k}|^2\big\|_\infty$ from \eqref{envest}, we obtain
		\begin{equation*}
			\begin{split}
				\sum_{U\in\mathcal{G}_{\tau_k}}\mu(U)\left(\fint_U\sum_{\theta\prec\tau_k}|f_\theta|^2\right)^2&\leq
                2^{\frac{r}{2}-2}e^{r-4}N^{r-4}(\#\tau_k)^{r-4}\lambda^{r-2}\alpha^{4-r}\\
                &\times\left(\max_{\gamma_k\prec\tau_k}\#(\gamma\prec\gamma_k)\times\#(\gamma\prec\tau_k)\right)^{\frac{r}{2}-1}\sum_{U\in\mathcal{G}_{\tau_k}}\int_U\sum_{\gamma_k\prec\tau_k}|f_{\gamma_k}|^2,
			\end{split}
		\end{equation*}
		and by local orthogonality and uniformity assumption (b)
		\begin{equation*}
			\sum_{U\in\mathcal{G}_{\tau_k}}\int_U\sum_{\gamma_k\prec\tau_k}|f_{\gamma_k}|^2=\int\sum_{\gamma\prec\tau_k}|f_\gamma|^2\leq\lambda^{2-r}\frak{C}_{r}\sum_{\gamma\prec\tau_k}\|f_\gamma\|_r^r.
		\end{equation*}
		Together we get the estimate
		\begin{equation*}
			\begin{split}
				\sum_{U\in\mathcal{G}_{\tau_k}}\mu(U)\left(\fint_U\sum_{\theta\prec\tau_k}|f_\theta|^2\right)^2&\leq \frak{C}_{r}\big[2^{1/2}eN(\#\tau_k)\alpha^{-1}\big]^{r-4}\\
				&\times\left(\max_{\gamma_k\prec\tau_k}\#(\gamma\prec\gamma_k)\times\#(\gamma\prec\tau_k)\right)^{\frac{r}{2}-1}\sum_{\gamma\prec\tau_k}\|f_\gamma\|_r^r,
			\end{split}
		\end{equation*}
		as claimed.
	\end{proof}

    The above amounts to a proof of Theorem \ref{smallcap}, in the special case that $f$ satisfies the assumed regularity properties. It remains to remove those assumptions, which we do now.

    We find it convenient to separate out a proof of a superlevel set estimate for small cap decoupling at the critical exponent pair. The full result, Theorem \ref{smallcap}, will follow by an elementary argument and interpolation.
    \begin{theorem}[Main estimate for critical exponents]\label{mainest_crit}
        Suppose $R\geq 16$ and $N\geq 10$. Let $f$ be Schwartz--Bruhat supported in $B_R$ with Fourier support in $\cal{N}_{R^{-1}}(\bb{P}^1)$, such that $\max_\theta\|f_\theta\|_\infty=1$. Let $(r,q)=(2+\frac{2}{\beta},\frac{2+2\beta^{-1}}{2\beta^{-1}-1})$. Then, for each $R^{-1/2}\leq\alpha\leq R^{1/2}$, we have the estimate
        \begin{equation}\label{critexponent}
            \begin{split}
			 \alpha^r\mu\big(\{x:|f(x)|>\alpha\}\big)|&\leq\frac{e^{4r-3}}{50\cdot 2^{3r/2}}\frac{\bf q^{-3}\bf{p}^{10}}{(\log\bf p)^{r+7}}(\log R)^{3r+10}\\
             &\times\max\big(R^{\beta(r-\frac{r}{q}-1)-1},R^{\beta(\frac{r}{2}-\frac{r}{q})}\big)\Big(\sum_\gamma\|f_\gamma\|_r^q\Big)^{\frac{r}{q}}.
            \end{split}
        \end{equation}
    \end{theorem}

    \begin{proof}[Proof of Theorem \ref{mainest_crit}]

		We may write
		\begin{equation*}
			f=\sum_{N^{-2}R^{-1/2}<\lambda\leq 1}\sum_{\substack{\gamma\in\Gamma_\beta(R^{-1})\\\|f_\gamma\|_\infty\in(e^{-1}\lambda,\lambda]}}f_\gamma+N^{-2}R^{-1/2}\eta,
		\end{equation*}
		where the $\lambda$ range over values of the form $\lambda_{k+1}=\lfloor e^{-1}\lambda_k\rfloor$, $\lambda_0=1$, and $\eta$ is Schwartz--Bruhat, supported on $B_R$, and uniformly bounded by $1$. We abbreviate
		\begin{equation*}
			\Gamma_\beta^\lambda(R^{-1})=\big\{\gamma\in\Gamma_\beta(R^{-1}):\|f_\gamma\|_\infty\in(e^{-1}\lambda,\lambda]\big\}.
		\end{equation*}
		
		Then, for each $\lambda$, consider the wave envelope expansion
		\begin{equation*}
			\sum_{\gamma\in\Gamma_\beta^\lambda(R^{-1})}f_\gamma=\sum_{\gamma\in\Gamma_\beta^\lambda(R^{-1})}\sum_U 1_Uf_\gamma,
		\end{equation*}
		where each $U$ has dimensions $\sim R^\beta\times R$ and has long edge parallel to $\mathbf{n}_{c_\gamma}$. Since $\gamma\in\Gamma_\beta^\lambda(R^{-1})$, there is some $U$ such that $\|1_Uf\|_\infty\in(e^{-1}\lambda,\lambda]$. If we write $\mathcal{U}_\lambda=\mathcal{U}_\lambda^\gamma$ for the set of $U$ for which $\|1_Uf_\gamma\|_\infty\in(e^{-1}\lambda,\lambda]$, then for all $\gamma\in\Gamma_\beta^{\lambda}(R^{-1})$
		\begin{equation*}
			e^{-r}(\#\mathcal{U}_\lambda)R^{1+\beta}\lambda^r\leq\Big\|\sum_{U\in\mathcal{U}_\lambda}1_Uf_\gamma\Big\|_r^r\leq(\#\mathcal{U}_\lambda)R^{1+\beta}\lambda^r,
		\end{equation*}
		and so
		\begin{equation}\label{pruncomp}
			\Big\|\frac{1}{\lambda}\sum_{U\in\mathcal{U}_\lambda}1_Uf_\gamma\Big\|_2^2\leq(\#\cal{U}_\lambda)R^{1+\beta}\leq e^r\Big\|\frac{1}{\lambda}\sum_{U\in\mathcal{U}_\lambda}1_Uf_\gamma\Big\|_r^r
		\end{equation}
        and
        \begin{equation}\label{pruncomp2}
            \Big\|\frac{1}{\lambda}\sum_{U\in\mathcal{U}_\lambda}1_Uf_\gamma\Big\|_r^r\leq(\#\cal{U}_\lambda^\gamma)R^{1+\beta}\leq e^2\Big\|\frac{1}{\lambda}\sum_{U\in\mathcal{U}_\lambda}1_Uf_\gamma\Big\|_2^2.
        \end{equation}
		For each $1\leq\mathfrak{t}\leq R$ and each $\lambda$, write $\Gamma_\beta^{\lambda;\mathfrak{t}}(R^{-1})$ to be the collection of $\gamma\in\Gamma_\beta^{\lambda}(R^{-1})$ such that $\#\mathcal{U}_\lambda^\gamma\in(e^{-1}\mathfrak{t},\frak{t}]$. Define for $\gamma\in\Gamma_\beta^\lambda(R^{-1})$
		\begin{equation*}
			g_\gamma^{(\lambda)}=\frac{1}{\lambda}\sum_{U\in\mathcal{U}_\lambda}1_Uf_\gamma
		\end{equation*}
		and
		\begin{equation*}
			g^{(\lambda,\mathfrak{t})}=\sum_{\gamma\in\Gamma_\beta^{\lambda;\mathfrak{t}}(R^{-1})} g_\gamma^{(\lambda)}.
		\end{equation*}
		We may observe that $g^{(\lambda,\frak{t})}$ is a Schwartz--Bruhat function with Fourier support in the $R^{-1}$-neighborhood of the truncated parabola. Thus, for each $\lambda,\mathfrak{t}$, and $\mathfrak{a}>0$ we have
		\begin{equation*}
			\mathfrak{a}^4\mu\big(\{x:|\lambda g^{(\lambda,\mathfrak{t})}(x)|>\mathfrak{a}\}\big)\leq CN^{E_2}\sum_{0\leq k\leq N}\sum_{\tau_k}\sum_{U\in\mathcal{G}_{\tau_k}[\lambda g^{(\lambda,\frak{t})};\mathfrak{a}]}\mu(U)\left(\fint_U\sum_{\theta\prec\tau_k}|\lambda g_\theta^{(\lambda,\mathfrak{t})}|^2\right)^2,
		\end{equation*}
		where we have written $\mathcal{G}_{\tau_k}[\lambda g^{(\lambda,\frak{t})};\mathfrak{a}]$ to record that the pruning is with respect to the function $\lambda g^{(\lambda,\frak{t})}$ with the amplitude parameter $\mathfrak{a}$. For each $0\leq k\leq N$ and each $1\leq\mathfrak{s}\leq R^{1/2}$, let $\mathcal{T}_k(\mathfrak{s})$ denote the collection of $\tau_k$ such that $\#\{\gamma\prec\tau_k:g_\gamma^{(\lambda,\mathfrak{t})}\neq 0\}\in(e^{-1}\mathfrak{s},\frak{s}]$. By pigeonholing, for each $k$ we may find $\mathfrak{s}^k$ such that
		\begin{equation*}
			\mathfrak{a}^4\mu\big(\{x:|\lambda g^{(\lambda,\mathfrak{t})}(x)|>\mathfrak{a}\}\big)\leq 2^{-1}CN^{E_2}(\log R)\sum_{0\leq k\leq N}\sum_{\tau_k\in\mathcal{T}_k(\mathfrak{s}^k)}\sum_{U\in\mathcal{G}_{\tau_k}[\lambda g^{(\lambda,\frak{t})};\mathfrak{a}]}\mu(U)\left(\fint_U\sum_{\theta\prec\tau_k}|\lambda g_\theta^{(\lambda,\mathfrak{t})}|^2\right)^2.
		\end{equation*}
		By Prop. \ref{partial} and \eqref{pruncomp}, we have
		\begin{equation*}
            \begin{split}
			\mathfrak{a}^r\mu\big(\{x:|\lambda g^{(\lambda,\mathfrak{t})}(x)|>\mathfrak{a}\}\big)&\leq C2^{\frac{r}{2}-3}e^{2r-4}N^{r+E_2-4}(\log R)\sum_{0\leq k\leq N}(\#\mathcal{T}_k(\mathfrak{s}^k))^{r-4}\\
            &\times\sum_{\tau_k\in\mathcal{T}_k(\mathfrak{s}^k)}\left(\mathfrak{s}^k\max_{\gamma_k\prec\tau_k}\#(\gamma\prec\gamma_k)\right)^{\frac{r}{2}-1}\sum_{\gamma\prec\tau_k}\|\lambda g_\gamma^{(\lambda,\mathfrak{t})}\|_r^r,
            \end{split}
		\end{equation*}
		and by pigeonholing to a single $0\leq k_*\leq N$ we have
		\begin{equation*}\label{refdprecases}
			\begin{split}
				\mathfrak{a}^p&\mu\big(\{x:|\lambda g^{(\lambda,\mathfrak{t})}(x)|>\mathfrak{a}\}\big)\\
				&\leq C2^{\frac{r}{2}-3}e^{2r-4}(N+1)N^{r+E_2-4}(\log R)(\#\mathcal{T}_{k_*}(\mathfrak{s}^{k_*}))^{r-4}\sum_{\tau_{k_*}\in\mathcal{T}_{k_*}(\mathfrak{s}^{k_*})}\left(\mathfrak{s}^{k_*}\max_{\gamma_{k_*}\prec\tau_{k_*}}\#(\gamma\prec\gamma_{k_*})\right)^{\frac{r}{2}-1}\sum_{\gamma\prec\tau_{k_*}}\|\lambda g_\gamma^{(\lambda,\mathfrak{t})}\|_r^r\\
				&\leq C2^{\frac{r}{2}-3}e^{2r-4}(N+1)N^{r+E_2-4}(\log R)(\#\mathcal{T}_{k_*}(\mathfrak{s}^{k_*}))^{r-3}\left(\mathfrak{s}^{k_*}\max_{\gamma_{k_*}\prec\tau_{k_*}}\#(\gamma\prec\gamma_{k_*})\right)^{\frac{r}{2}-1}\mathfrak{s}^{k_*}\lambda^r\mathfrak{t}\cdot R^{1+\beta}.
			\end{split}
		\end{equation*}
        We claim that this amounts to a decoupling inequality for the $\lambda g^{(\lambda,\frak{t})}$. That is,
        \begin{lemma}\label{param_dec_lem} For each $\lambda,\frak{t}$, and $\frak{a}$, and each $r\geq 4$, we have the estimate
            \begin{equation}\label{pigeonholed}
                \mathfrak{a}^r\mu\big(\{x:|\lambda g^{(\lambda,\mathfrak{t})}(x)|>\mathfrak{a}\}\big)\leq C2^{\frac{r}{2}-3}e^{4r-3}(N+1)N^{r+E_2-4}(\log R)\max\big(R^{\beta(r-\frac{r}{q}-1)-1},R^{\beta(\frac{r}{2}-\frac{r}{q})}\big)\Big(\sum_\gamma\|\lambda g_\gamma^{(\lambda,\mathfrak{t})}\|_r^q\Big)^{\frac{r}{q}}.
            \end{equation}
        \end{lemma}
        We postpone the proof until the end of the argument, to preserve the flow of our calculation. Recalling the identity
		\begin{equation*}
			f=\sum_{N^{-2}R^{-1/2}<\lambda\leq 1}\lambda g^{(\lambda)}+N^{-2}R^{-1/2}\eta,
		\end{equation*}
		and consequently, for a suitable $\lambda_*$,
		\begin{equation*}
			\begin{split}
				\alpha^r\mu\big(\{x:|f(x)|>\alpha\}\big)&\leq\alpha^r\sum_{N^{-2}R^{-1/2}<\lambda\leq 1}\mu\Big(\Big\{x:|\lambda g^{(\lambda)}(x)|\geq\frac{\alpha}{\cal{Z}}\Big\}\Big)\\
				&\leq\cal{Z}^{r+1}\Big(\frac{\alpha}{\cal{Z}}\Big)^r\mu\Big(\Big\{x:|\lambda_*g^{(\lambda_*)}(x)|\geq\frac{\alpha}{\cal{Z}}\Big\}\Big),
			\end{split}
		\end{equation*}
		(where we have abbreviated $\cal{Z}=2\log N+\frac{1}{2}\log R+1$) and hence, for a suitable $\mathfrak{t}$,
		\begin{equation*}
			\alpha^r\mu\big(\{x:|f(x)|>\alpha\}\big)\leq\left[\cal{Z}\log R\right]^{r+1}\Big(\frac{\alpha}{\cal{Z}\log R}\Big)^r\mu\Big(\Big\{x:|\lambda_*g^{(\lambda_*,\mathfrak{t})}(x)|\geq\frac{\alpha}{\cal{Z}\log R}\Big\}\Big),
		\end{equation*}
		which by \eqref{pigeonholed}, applied to $\mathfrak{a}=\frac{\alpha}{\cal{Z}\log R}$, implies
		\begin{equation*}
            \begin{split}
			 \alpha^r\mu\big(\{x:|f(x)|>\alpha\}\big)|&\leq C2^{\frac{r}{2}-3}e^{4r-3}\cal{Z}^{r+1}(N+1)N^{r+E_2-4}(\log R)^{r+2}\\
    &\times\max\big(R^{\beta(r-\frac{r}{q}-1)-1},R^{\beta(\frac{r}{2}-\frac{r}{q})}\big)\Big(\sum_\gamma\|\lambda_* g_\gamma^{(\lambda_*,\mathfrak{t})}\|_r^q\Big)^{\frac{r}{q}}.
            \end{split}
		\end{equation*}
		Finally, we note that each $\lambda_*g_\gamma^{(\lambda_*,\mathfrak{t})}$ is obtained by taking a subsum of a partition of unity applied to $f_\gamma$, so we conclude that
		\begin{equation}\label{superlevel_last}
            \begin{split}
			\alpha^r\mu\big(\{x:|f(x)|>\alpha\}\big)|&\leq C2^{\frac{r}{2}-3}e^{4r-3}\cal{Z}^{r+1}(N+1)N^{r+E_2-4}(\log R)^{r+2}\\
            &\times\max\big(R^{\beta(r-\frac{r}{q}-1)-1},R^{\beta(\frac{r}{2}-\frac{r}{q})}\big)\Big(\sum_\gamma\|f_\gamma\|_r^q\Big)^{\frac{r}{q}},
            \end{split}
		\end{equation}
        as claimed.
        
    \end{proof}
	
	\begin{proof}[Proof of Theorem \ref{smallcap}]

        By scaling we may assume that $\max_\theta\|f_\theta\|_\infty=1$. By \cite{johnsrude2024restricted}, Lemma 5.4, we may assume that $f$ is supported in $B_R$. By the layer-cake integral
        \begin{equation*}
            \int|f|^r=r\int_0^{R^{1/2}}\alpha^{r-1}\mu(U_\alpha)d\alpha,
        \end{equation*}
        and the inequalities
        \begin{equation}\label{lowamptrivial}
            R^{-\frac{r}{2}+2}\leq R^{\frac{r}{2}+1-\beta(r-1)}\leq\max_\gamma\|f_\gamma\|_r^r,
        \end{equation}
        \begin{equation}\label{scalechoice}
            \int_{R^{-1/2}}^{R^{1/2}}r\alpha^{r-1}\mu(U_\alpha)d\alpha\leq r(\log R)\sup_{\alpha\in [R^{-1/2},R^{1/2}]}\alpha^r\mu(U_\alpha),
        \end{equation}
        we transform \eqref{superlevel_last} into the decoupling estimate
        \begin{equation*}
            D_{r_\beta,q_\beta}^\K(R;\beta)\leq \frac{e^{5+2\beta^{-1}}}{90\cdot 2^{3\beta^{-1}}}\frac{\bf q^{-3}\bf{p}^{10}}{(\log\bf p)^{9+2\beta^{-1}}}(\log R)^{17+6\beta^{-1}}R^{\beta(\frac{r_\beta}{2}-\frac{r_\beta}{q_\beta})}.
        \end{equation*}

        It remains to use interpolation to cover the remaining exponents. The two quantities we will compare against are
        \begin{equation*}
            D_{\infty,1}^\K(R;\beta)\leq 1
        \end{equation*}
        and
        \begin{equation*}
            D_{4,4}^\K(R;\beta)\leq 2R^\beta.
        \end{equation*}
        The latter estimate follows trivially from C\'ordoba--Fefferman and flat decoupling.

        For each $3\leq r\leq\infty$, write $\frac{1}{q_r}=1-\frac{3}{r}$. Write $r_\beta=2+\frac{2}{\beta}$ and $q_\beta=q_{r_\beta}$. Let $\frac{3}{r}+\frac{1}{q}\leq 1$. Suppose that $2+\frac{2}{\beta}\leq r$ and $\beta<1$. Then, from the two inequalities
        \begin{equation*}
            D_{p_\beta,q_\beta}^\K(R;\beta)\leq \frac{e^{5+2\beta^{-1}}}{90\cdot 2^{3\beta^{-1}}}\frac{\bf q^{-3}\bf{p}^{10}}{(\log\bf p)^{9+2\beta^{-1}}}(\log R)^{17+6\beta^{-1}}R^{\beta(\frac{r_\beta}{2}-\frac{r_\beta}{q_\beta})},
        \end{equation*}
        \begin{equation*}
            D_{\infty,1}^\K(R;\beta)\leq 1,
        \end{equation*}
        we interpolate to obtain
        \begin{equation*}
            D_{r,q_r}^\K(R;\beta)\leq \frac{e^{5+2\beta^{-1}}}{90\cdot 2^{3\beta^{-1}}}\frac{\bf q^{-3}\bf{p}^{10}}{(\log\bf p)^{9+2\beta^{-1}}}(\log R)^{17+6\beta^{-1}}R^{\beta(r-\frac{r}{q_r}-1)-1}.
        \end{equation*}
        By H\"older, we conclude the desired
        \begin{equation*}
            D_{r,q}^\K(R;\beta)\leq \frac{e^{5+2\beta^{-1}}}{90\cdot 2^{3\beta^{-1}}}\frac{\bf q^{-3}\bf{p}^{10}}{(\log\bf p)^{9+2\beta^{-1}}}(\log R)^{17+6\beta^{-1}}R^{\beta(r-\frac{r}{q}-1)-1}.
        \end{equation*}
        Next, suppose that $4<r\leq 2+\frac{2}{\beta}$. We may assume $\beta<1$. Then, from the inequalities
        \begin{equation*}
            D_{r_\beta,q_\beta}^\K(R;\beta)\leq \frac{e^{5+2\beta^{-1}}}{90\cdot 2^{3\beta^{-1}}}\frac{\bf q^{-3}\bf{p}^{10}}{(\log\bf p)^{9+2\beta^{-1}}}(\log R)^{17+6\beta^{-1}}R^{\beta(\frac{r_\beta}{2}-\frac{r_\beta}{q_\beta})},
        \end{equation*}
        \begin{equation*}
            D_{4,4}^\K(R;\beta)\leq 2R^\beta
        \end{equation*}
        we interpolate to obtain
        \begin{equation*}
            D_{r,q_r}^\K(R;\beta)\leq \frac{e^{5+2\beta^{-1}}}{90\cdot 2^{3\beta^{-1}}}\frac{\bf q^{-3}\bf{p}^{10}}{(\log\bf p)^{9+2\beta^{-1}}}(\log R)^{17+6\beta^{-1}}R^{\beta(\frac{r}{2}-\frac{r}{q_r})}.
        \end{equation*}
        By H\"older, we obtain the desired
        \begin{equation*}
            D_{r,q}^\K(R;\beta)\leq \frac{e^{5+2\beta^{-1}}}{90\cdot 2^{3\beta^{-1}}}\frac{\bf q^{-3}\bf{p}^{10}}{(\log\bf p)^{9+2\beta^{-1}}}(\log R)^{17+6\beta^{-1}}R^{\beta(\frac{r}{2}-\frac{r}{q})}.
        \end{equation*}
        It remains to consider the regime $3\leq r\leq 4$. Then, from the estimates
        \begin{equation*}
            D_{4,4}^\K(R;\beta)\leq 2R^\beta,
        \end{equation*}
        \begin{equation*}
            D_{2,2}^\K(R;\beta)=1,
        \end{equation*}
        we obtain the estimate
        \begin{equation*}
            D_{3,3}^\K(R;\beta)\leq \sqrt{2}R^{\frac{\beta}{2}},
        \end{equation*}
        whence
        \begin{equation*}
            D_{3,\infty}^\K(R;\beta)\leq \sqrt{2}R^{\frac{3\beta}{2}}.
        \end{equation*}
        By interpolation and H\"older, we conclude
        \begin{equation*}
            D_{r,q}^\K(R;\beta)\leq 2R^{\beta(\frac{r}{2}-\frac{r}{q})}.
        \end{equation*}

        In all, we conclude the estimate
        \begin{equation*}
            D_{r,q}^\K(R;\beta)\leq \frac{e^{5+2\beta^{-1}}}{90\cdot 2^{3\beta^{-1}}}\frac{\bf q^{-3}\bf{p}^{10}}{(\log\bf p)^{9+2\beta^{-1}}}(\log R)^{17+6\beta^{-1}}\max\Big(R^{\beta(r-\frac{r}{q}-1)-1},R^{\beta(\frac{r}{2}-\frac{r}{q})}\Big).
        \end{equation*}
        The result follows.

        \end{proof}

        To close the argument, it remains only to prove Lemma \ref{param_dec_lem}.
        
        \begin{proof}[Proof of Lemma \ref{param_dec_lem}]
    
            The proof is by casework on $r$ and $k$. Suppose $R_k\geq R^{1-\beta}$. By a straightforward calculation (see Case 2 of the proof of Theorem 5 in \cite{GM1}),
            \begin{equation*}
                (\#\mathcal{T}_{k_*}(\mathfrak{s}^{k_*}))^{r-4}\left(\mathfrak{s}^{k_*}\max_{\substack{\gamma_{k_*}\prec\tau_{k_*}\\\tau_{k_*}\in\mathcal{T}_{k_*}(\mathfrak{s}^{k_*})}}\#(\gamma\prec\gamma_{k_*})\right)^{\frac{r}{2}-1}\leq R^{\beta(r-\frac{r}{q}-1)-1}\big(\mathfrak{s}^{k_*}\times\#\mathcal{T}_{k_*}(\mathfrak{s}^{k_*})\big)^{\frac{r}{q}-1}.
            \end{equation*}
    		Consequently,
    		\begin{equation}
    			\begin{split}
    				\mathfrak{a}^r&\mu\big(\big\{x:|\lambda g^{(\lambda,\mathfrak{t})}(x)|>\mathfrak{a}\big\}\big)\\
    				&\leq C2^{\frac{r}{2}-3}e^{2r-4}(N+1)N^{r+E_2-4}(\log R)R^{\beta(r-\frac{r}{q}-1)-1}\big(\#\mathcal{T}_{k_*}(\mathfrak{s}^{k_*})\big)^{\frac{r}{q}}(\mathfrak{s}^{k_*})^{\frac{r}{q}}\lambda^r\mathfrak{t}R^{1+\beta}\\
    				&\leq C2^{\frac{r}{2}-3}e^{4r-3}(N+1)N^{r+E_2-4}(\log R)R^{\beta(r-\frac{r}{q}-1)-1}\Big(\sum_\gamma\|\lambda g_\gamma^{(\lambda,\mathfrak{t})}\|_r^q\Big)^{\frac{r}{q}}.
    			\end{split}
    		\end{equation}
            Suppose instead that $R_k\leq R^{1-\beta}$, and 
            $2+\frac{2}{\beta}\leq r\leq 6$. Then, from the inequality
            \begin{equation*}
                1\leq R^{\beta(\frac{r}{2}-1)}R^{-1},
            \end{equation*}
            it follows that for any $\gamma_k$ we have
            \begin{equation*}
                [\#(\gamma\prec\gamma_k)]^{\frac{r}{2}-1}\leq R^{-1}R^{\beta(\frac{r}{2}-1)}R_k^{3-\frac{r}{2}}.
            \end{equation*}
            Rearranging, we have
            \begin{equation*}
                R^{\beta(r-\frac{r}{q}-3)}[\#(\gamma\prec\gamma_k)]^{\frac{r}{2}-1}\leq R^{\beta(r-\frac{r}{q}-1)-1}\Big(R^{-\beta}R_k\Big)^{3-\frac{r}{2}},
            \end{equation*}
            which implies (using $\frac{1}{q}+\frac{3}{r}\leq 1$)
            \begin{equation*}
                (\mathfrak{s}^{k_*}\times\#\mathcal{T}_{k_*}(\mathfrak{s}^{k_*}))^{r-3}[\#(\gamma\prec\gamma_k)]^{\frac{r}{2}-1}\leq R^{\beta(r-\frac{r}{q}-1)-1}\Big(R^{-\beta}R_k\Big)^{3-\frac{r}{2}}(\#\gamma)^{\frac{r}{q}}.
            \end{equation*}
            Note that, for each $\tau_k\in\cal{T}_{k_*}(\frak{s}^{k_*})$,
            \begin{equation*}
                (\#\cal{T}_{k_*}(\frak{s}^{k_*}))^{r-4}(\#\gamma\prec\tau_k)^{\frac{r}{2}-1}\leq(\frak{s}^{k_*})^{3-\frac{r}{2}}(\mathfrak{s}^{k_*}\times\#\mathcal{T}_{k_*}(\mathfrak{s}^{k_*}))^{r-4},
            \end{equation*}
            so that
            \begin{equation*}
                \mathfrak{s}^{k_*}(\#\cal{T}_{k_*}(\frak{s}^{k_*}))^{r-3}(\#\gamma\prec\tau_k)^{\frac{r}{2}-1}[\#(\gamma\prec\gamma_k)]^{\frac{r}{2}-1}\leq R^{\beta(r-\frac{r}{q}-1)-1}(\#\gamma)^{\frac{r}{q}}.
            \end{equation*}
            Thus, we achieve the estimate
            \begin{equation*}
                \begin{split}
                    \mathfrak{a}^r\mu\big(\{x:|\lambda g^{(\lambda,\mathfrak{t})}(x)|>\mathfrak{a}\}\big)&\leq C2^{\frac{r}{2}-3}e^{2r-4}(N+1)N^{r+E_2-4}(\log R)R^{\beta(r-\frac{r}{q}-1)-1}(\#\gamma)^{\frac{r}{q}}\lambda^r\mathfrak{t}\cdot R^{1+\beta}\\
                    &\leq C2^{\frac{r}{2}-3}e^{4r-3}(N+1)N^{r+E_2-4}(\log R)R^{\beta(r-\frac{r}{q}-1)-1}\Big(\sum_\gamma\|\lambda g_\gamma^{(\lambda,\frak{t})}\|_r^q\Big)^{\frac{r}{q}}.
                \end{split}
            \end{equation*}
    
            It remains to consider the case $R_k\leq R^{1-\beta}$, $4\leq r\leq 2+\frac{2}{\beta}$. In this case, from the inequalities
            \begin{equation*}
                R_k^{\frac{r}{2}-3}\leq 1,\quad \frak{s}^{k_*}\leq RR_k^{-1},\quad\#\cal{T}_{k_*}(\frak{s}^{k_*})\leq R_k,
            \end{equation*}
            we conclude the inequality
            \begin{equation*}
                (\frak{s}^{k_*})^{\frac{r}{2}-\frac{r}{q}}(\#\cal{T}_{k_*}(\frak{s}^{k_*}))^{r-\frac{r}{q}-3}\leq R^{\beta(\frac{r}{2}-\frac{r}{q})}.
            \end{equation*}
            Thus,
            \begin{equation*}
                (\frak{s}^{k_*})^{\frac{r}{2}}(\#\cal{T}_{k_*}(\frak{s}^{k_*}))^{r-3}\leq R^{\beta(\frac{r}{2}-\frac{r}{q})}(\#\gamma)^{\frac{r}{q}},
            \end{equation*}
            and we conclude that
            \begin{equation*}
                \begin{split}
                    \mathfrak{a}^r\mu\big(\{x:|\lambda g^{(\lambda,\mathfrak{t})}(x)|>\mathfrak{a}\}\big)&\leq C2^{\frac{r}{2}-3}e^{2r-4}(N+1)N^{r+E_2-4}(\log R)R^{\beta(\frac{r}{2}-\frac{r}{q})}(\#\gamma)^{\frac{r}{q}}\lambda^r\mathfrak{t}\cdot R^{1+\beta}\\
                    &\leq C2^{\frac{r}{2}-3}e^{4r-3}(N+1)N^{r+E_2-4}(\log R)R^{\beta(\frac{r}{2}-\frac{r}{q})}\Big(\sum_\gamma\|\lambda g_\gamma^{(\lambda,\frak{t})}\|_r^q\Big)^{\frac{r}{q}}.
                \end{split}
            \end{equation*}
                
        \end{proof}

    \section{Transference and real integrals}\label{section:transfer}

    In this section, we supply the tools needed to convert our main results (Theorems \ref{smallcap} and \ref{squarefunction}) into the applications described in subsection \ref{subsec:application}. The basic procedure was already described in our previous work \cite{johnsrude2025sparse}; we give a simplified account here to be self-contained.

    We begin by sketching a proof of Theorem \ref{thm:sparse_mv}. It will contain many of the details needed to explain the other two applications.

    \begin{proof}[Proof sketch of Theorem \ref{thm:sparse_mv}]

        We will first handle the case $\gamma=0$, and then explain the needed modifications for the general $\gamma$ case.
        
        Let $N\in\N$ and $\sigma\in[0,1]$ be such that $N,N^\sigma\in p^\Z$ with $p$ an odd prime. Fix $\K=\Q_p$. Consider the domain $\frak D\subseteq\Q_p^2$ defined by
        \begin{equation*}
            \frak D=\O\times B(0,N^{\sigma-1}).
        \end{equation*}
        Thus, $\frak D$ consists of those $(x,y)\in\Q_p^2$ with $|x|\leq 1$, $|y|\leq N^{\sigma-1}$. By appealing to Theorem \ref{smallcap}, we may show
        \begin{equation}\label{ineq:nonarch_mv}
            \fint_{\frak D}\Big|\sum_{n=0}^{N-1}a_ne_{\Q_p}\Big(\frac{n}{N}x+\frac{n^2}{N^2}t\Big)\Big|^rdxdt\lesssim_\bf p(\log N)^{O(1)}\big(N^{(r-\frac{r}{q}-1)-1-\sigma}+N^{\frac{r}{2}-\frac{r}{q}}\big)\left(\sum_{n=0}^{N-1}|a_n|^q\right)^{r/q}.
        \end{equation}
        Here we have written $e=e_{\Q_p}$ for the function specified in \eqref{id:padic_char}. The bound \eqref{ineq:nonarch_mv} follows from standard changes-of-variable and periodicity, together with an application of small cap decoupling \ref{smallcap}.
    
        Next, note that
        \begin{equation*}
            f(x,t)=\sum_{n=0}^{N-1}a_ne_{\Q_p}\Big(\frac{n}{N}x+\frac{n^2}{N^2}t\Big)
        \end{equation*}
        has the local constancy property $f(x+h,t+k)=f(x,t)$ whenever $|h|\leq N^{-1}, |t|\leq N^{-2}$. In particular, our integral may be written as the following Riemann sum:
        \begin{equation*}
            \begin{split}
                \fint_{\frak D}|f(x,t)|^rdxdt&=N^{-2-\sigma}\sum_{a=0}^{N-1}\sum_{b=0}^{N^{1+\sigma}-1}|f(a,N^{1-\sigma}b)|^r\\
                &=N^{-2-\sigma}\sum_{a=0}^{N-1}\sum_{b=0}^{N^{1+\sigma}-1}\Bigg|\sum_{n=0}^{N-1}a_ne_{\Q_p}\Big(\frac{na}{N}+\frac{n^2b}{N^{1+\sigma}}\Big)\Bigg|^r.
            \end{split}
        \end{equation*}
        Each of the arguments in the latter $e_{\Q_p}$ is a rational number whose denominator is a power of $p$. Thus, after inspecting \eqref{id:padic_char}, we may freely substitute $e_{\Q_p}$ for $e_\R$:
        \begin{equation*}
            \fint_{\frak D}|f(x,t)|^rdxdt=N^{-2-\sigma}\sum_{a=0}^{N-1}\sum_{b=0}^{N^{1+\sigma}-1}\Bigg|\sum_{n=0}^{N-1}a_ne_{\R}\Big(\frac{na}{N}+\frac{n^2b}{N^{1+\sigma}}\Big)\Bigg|^r.
        \end{equation*}
        Here of course $e_\R$ is the usual Euclidean phase $e_\R(x)=\exp(2\pi ix)$. In particular, we may write
        \begin{equation*}
            \fint_{\frak D}|f(x,t)|^rdxdt=\int\Bigg|\sum_{n=0}^{N-1}a_ne_{\R}\Big(nx+n^2t\Big)\Bigg|^rd\Delta(x,t),
        \end{equation*}
        where we have written $\Delta$ for the measure on $\R^2$ defined by
        \begin{equation*}
            \Delta=N^{-2-\sigma}\sum_{a=0}^{N-1}\sum_{b=0}^{N^{1+\sigma}-1}\delta_{\frac{a}{N}}\otimes\delta_{\frac{b}{N^{1+\sigma}}}.
        \end{equation*}
        Here $\delta_c$ is the Dirac delta mass located at $c\in\R$. Since translations $(x,t)\mapsto(x+\frak{x},t+\frak{t})$ amount to replacing the coefficients $a_n$ with unimodular multiples $\tilde{a}_n$, $|a_n|=|\tilde{a}_n|$, we may write
        \begin{equation*}
            \fint_{A(N,\sigma;0)}\Bigg|\sum_{n=1}^Na_ne\Big(nx+n^2t\Big)\Bigg|^rdxdt\leq\max_{\substack{\{\tilde a_n\}_n\\|\tilde a_n|=|a_n|}}\int\Bigg|\sum_{n=0}^{N-1}\tilde a_ne\Big(nx+n^2t\Big)\Bigg|^rd\Delta(x,t).
        \end{equation*}
        In particular, we obtain the estimate
        \begin{equation*}
            \fint_{A(N,\sigma;0)}\Bigg|\sum_{n=1}^Na_ne\Big(nx+n^2t\Big)\Bigg|^rdxdt\lesssim_{\bf p}(\log N)^{O(1)}\big(N^{(r-\frac{r}{q}-1)-1-\sigma}+N^{\frac{r}{2}-\frac{r}{q}}\big)\left(\sum_{n=0}^{N-1}|a_n|^q\right)^{r/q}.
        \end{equation*}

        We have justified Theorem \ref{thm:sparse_mv} when $\gamma=0$. We now discuss the intermediate case $\gamma\in(0,1)$; recall that the $\gamma=1$ case is the usual ``short mean value estimate'' \`a la Theorem \ref{thm:short_mv}, and the logarithmic factor for (real) small cap decoupling was obtained in \cite{johnsrude2023small}. The basic point is to consider exponential sums of the form
        \begin{equation*}
            f(x,t,x',t')=\sum_{n=1}^Na_ne_{\Q_p}(nx+n^2t)e_{\R}\Big(\frac{n}{N}x'+\frac{n^2}{N^2}t'\Big),
        \end{equation*}
        where $x,t\in\Q_p$ and $x',t'\in\R$. Thus, $f$ is a complex-valued function defined on the LCA group $\Q_p\times\R$; the general theory of Pontryagin duality implies that the tools of Fourier analysis on the $\Q_p$ and $\R$ factors may be freely lifted to the product.

        We supply several of the technical steps. Let
        \begin{equation*}
            g(x,t,x',t')=\sum_{n=1}^Na_n1_{N^{-1}\Z_p^2}(x,t)\eta_{B_N}(x',t')e_{\Q_p}(nx+n^2t)e_{\R}\Big(\frac{n}{N}x'+\frac{n^2}{N^2}t'\Big),
        \end{equation*}
        where $\eta_{B_N}$ is a nonnegative smooth function satisfying $\eta_{B_N}\gtrsim 1$ on $B_N\subseteq\R^2$ and $\hat\eta_{B_N}$ is supported in $B_{N^{-1}}$. Let $g_n$ be the $n$th summand in the definition of $g$. By using $p$-adic decoupling, Minkowski, Fubini, and real decoupling in turn,
        \begin{equation*}
            \begin{split}
                \int_{\R^2}&\int_{\Q_p^2}|g(x,t,x',t')|^rdxdtdx'dt'\\
                &\leq p^7B(N^{1-\gamma}) \int_{\R^2}\left(\sum_{\tau}\left(\int_{\Q_p^2}\Bigg|\sum_{n\in\tau}g_n(x,t,x',t')\Bigg|^rdxdt\right)^{q/r}\right)^{r/q}dx'dt'\\
                &\leq p^7B(N^{1-\gamma})\left(\sum_\tau\left(\int_{\Q_p^2}\int_{\R^2}\Bigg|\sum_{n\in\tau}g_n(x,t,x',t')\Bigg|^rdx'dt'dxdt\right)^{q/r}\right)^{r/q}\\
                &\leq p^7B(N^{1-\gamma})B(N^\gamma)\left(\sum_\tau\sum_{\tau'}\left(\int_{\Q_p^2}\int_{\R^2}\Bigg|\sum_{n\in\tau\cap\tau'}g_n(x,t,x',t')\Bigg|^rdx'dt'dxdt\right)^{q/r}\right)^{r/q}.
            \end{split}
        \end{equation*}
        Here we have the notation
        \begin{equation*}
            B(M)=C(\log M)^{C'}\big(M^{(r-\frac{r}{q}-1)-\sigma-1}+M^{\frac{r}{2}-\frac{r}{q}}\big),
        \end{equation*}
        with $C,C'\gg 1$ constants arising from the decoupling theorem \ref{smallcap} and the real analogue in \cite{johnsrude2023small}. The sets $\tau$ range over $\cal P(\O,N^{\gamma-1})$ and the sets $\tau'$ range over the partition of $[0,N]$ into intervals of length $N^{1-\gamma}$. It follows that the intersections $\tau\cap\tau'$ are singletons, so
        \begin{equation*}
            \int_{\R^2}\int_{\Q_p^2}|g(x,t,x',t')|^rdxdtdx'dt'\leq p^7 B(N)N^{4}\left(\sum_{n=1}^N|a_n|^q\right)^{r/q}.
        \end{equation*}
        Applying the same rescaling, transference, and translation principles as before, we obtain the estimate
        \begin{equation*}
            \fint_{A(N,\sigma;\gamma)}\Bigg|\sum_{n=1}^Ne(nx+n^2t)\Bigg|^rdxdt\lesssim p^7(\log N)^{O(1)}\big(N^{(r-\frac{r}{q}-1)-1-\sigma}+N^{\frac{r}{2}-\frac{r}{q}}\big)\left(\sum_{n=1}^N|a_n|^q\right)^{r/q},
        \end{equation*}
        as desired.
        
    \end{proof}

    We sketch the argument for Theorem \ref{thm:gaussian_mv} next, which is similar in spirit to the previous result.

    \begin{proof}[Proof of Theorem \ref{thm:gaussian_mv}]
        We work over $\K=\Q_3\big[\sqrt{-1}\big]$. A uniformizer $\omega$ for $\K$ is $3$; here, we have used that $p\equiv 3$ mod $4$. Our choice of basic character is
        \begin{equation*}
            e_{\K}(x)=e_{\Q_3}\big(\mathrm{Tr}_{\K/\Q_3}(x)\big).
        \end{equation*}
        This $e_{\K}$ satisfies the necessary condition $e_{\K}(\O)=\{1\}\neq e_\K(3^{-1}\O)$. We also have
        \begin{equation*}
            e_{\K}\big(a+b\sqrt{-1}\big)=e_{\Q_3}(a),\quad a,b\in\Q_3.
        \end{equation*}
        We study the same mean-value estimate considered in the proof of Theorem \ref{thm:sparse_mv}, framed over $\K$ instead of $\Q_p$. We have
        \begin{equation*}
            \begin{split}
                \fint_{\O\times B(0,N^{\sigma-1})}&\Bigg|\sum_{n,m=1}^Na_{n,m}e_{\K}\Big(\frac{n+m\sqrt{-1}}{N}x+\frac{(n+m\sqrt{-1})^2}{N^2}t\Big)\Bigg|^r\\
                &\lesssim(\log N)^{O(1)}\big(N^{(r-\frac{r}{q}-1)-1-\sigma}+N^{\frac{r}{2}-\frac{r}{q}}\big)\left(\sum_{n,m=1}^N|a_{n,m}|^q\right)^{r/q}
            \end{split}
        \end{equation*}
        As in the previous transference argument, we may reduce the integral to a Riemann sum. Here, we select $x$ to run over $a+b\sqrt{-1}$, $0\leq a,b\leq N-1$, and $t$ to run over $c+d\sqrt{-1}$, $0\leq c,d\leq N^{1+\sigma}-1$. Using the choice of $e_\K$, we see that
        \begin{equation*}
            e_{\K}\Big(\frac{n+m\sqrt{-1}}{N}x+\frac{(n+m\sqrt{-1})^2}{N^2}t\Big)=e_\R\Big(\frac{na}{N}-\frac{mb}{N}+\frac{(n^2-m^2)c}{N^2}-\frac{2nmd}{N^2}\Big).
        \end{equation*}
        The remainder of the transference goes through without substantial modification.
    \end{proof}

    We end by giving a proof of Theorem \ref{thm:dist}.

    \begin{proof}[Proof of Theorem \ref{thm:dist}]
        Define the function
        \begin{equation*}
            g(x,t)=1_{N^{-1-\sigma}\Z_p^2}(x,t)\sum_{j\in S}a_je_{\Q_p}\Big(jx+j^2t\Big).
        \end{equation*}
        Here $\Q_p$ is determined by the choice of $p$ in the statement of the theorem. Thus, $g$ is a Schwartz--Bruhat function on $\Q_p^2$.

        By Theorem \ref{squarefunction}, we have
        \begin{equation*}
            \mu(\{x\in\Q_p^2:|g(x)|>\alpha\})\leq\frac{1}{50} p^7(\log_pN)^{10}\sum_{\substack{k\in p^\Z\\N^{-\frac{1+\sigma}{2}}\leq p^{-k}\leq 1}}\sum_{\substack{\tau\\\mathrm{diam}(\tau)=p^{-k}}}\sum_{U\in\cal G_\tau[g;\alpha]}\mu(U)\left(\fint_U\sum_{\theta\prec\tau}|g_\theta|^2\right)^2,
        \end{equation*}
        where the $\theta$ are the caps of diameter $N^{-\frac{1+\sigma}{2}}$. By Proposition \ref{partial}, for each such $k$ and $\tau$ we have
        \begin{equation*}
            \sum_{U\in\cal G_\tau[g;\alpha]}\mu(U)\left(\fint_U\sum_{\theta\prec\tau}|g_\theta|^2\right)^2\leq\max_{\gamma_k\prec\tau}\#(S\cap\gamma_k)\#(S\cap\tau)\sum_{j\in S\cap\tau} N^{2+2\sigma}|a_j|^4,
        \end{equation*}
        where $\gamma_k$ are caps of diameter $p^kN^{-1-\sigma}$. Summing over $\tau$, we obtain
        \begin{equation}\label{ineq:padic_reg}
            \mu(\{x\in\Q_p^2:|g(x)|>\alpha\})\leq\frac{1}{100}p^7(\log_pN)^{11}\alpha^{-4}N^{2+2\sigma}\max_k\Big(\lambda_k(S)\cdot\lambda_{(1+\sigma)\log_pN-k}(S)\Big)\sum_{j\in S}|a_j|^4.
        \end{equation}
        It remains to relate $\mu(\{x\in\Q_p^2:|g(x)|>\alpha\})$ to $\mathcal{L}^2\big(U_\alpha\cap A(N,\sigma;0)\big)$. If $0\leq a\leq N-1$ and $0\leq b\leq N^{1+\sigma}-1$ are such that $(\frac{a}{N},\frac{b}{N^{1+\sigma}})\in\{x\in\Q_p^2:|g(x)|>\alpha\}$, then
        \begin{equation*}
            \left|\sum_{j\in S}a_je_\R\Big(j\frac{a}{N}+j^2\frac{b}{N^{1+\sigma}}\Big)\right|>\alpha.
        \end{equation*}
        Let
        \begin{equation*}
            \cal A=\Big\{0\leq a\leq N-1,0\leq b\leq N^{1+\sigma}-1:\big(\frac{a}{N},\frac{b}{N^{1+\sigma}}\big)\in\{x\in\Q_p^2:|g(x)|>\alpha\}\Big\}.
        \end{equation*}
        Then, by local constancy and periodicity, we find that
        \begin{equation*}
            \mu(\{x\in\Q_p^2:|g(x)|>\alpha\})= N^\sigma\#\cal A.
        \end{equation*}

        To get back to $U_\alpha$, we will need to study random perturbations of our main estimate. For $(\frak x,\frak t)\in[0,N^{-1}]\times[0,N^{-2}]$, define $a_j^{\frak x,\frak t}$ by
        \begin{equation*}
            a_j^{\frak x,\frak t}=a_je_\R(-j\frak x-j^2\frak t).
        \end{equation*}
        Define $g^{\frak x,\frak t}$ by
        \begin{equation*}
            g^{\frak x,\frak t}(x,t)=\sum_{j\in S} a_j^{\frak x,\frak t}e_{\Q_p}\Big(jx+j^2t\Big),
        \end{equation*}
        and correspondingly set
        \begin{equation*}
            \mathcal{A}^{\frak x,\frak t}=\Big\{0\leq a\leq N-1,0\leq b\leq N^{1+\sigma}-1:\big(\frac{a}{N},\frac{b}{N^{1+\sigma}}\big)\in\{x\in\Q_p^2:|g^{\frak x,\frak t}(x)|>\alpha\}\Big\}.
        \end{equation*}
        Note that $|a_j^{\frak x,\frak t}|=|a_j|$ for all $\frak x,\frak t$. In particular, \eqref{ineq:padic_reg} holds for all $\frak x,\frak t$ with the same upper bound, ranging over the distinct $g^{\frak x,\frak t}$.

        We integrate:
        \begin{equation*}
            \begin{split}
            \mathcal{L}^2\big(U_\alpha\cap A(N,\sigma;\gamma)\big)&=\int_{\substack{\frak x\in[0,N^{-1}]\\\frak t\in[0,N^{-2}]}}\sum_{a=0}^{N-1}\sum_{b=0}^{N^{1+\sigma}-1}1_{U_\alpha}\Big(\frac{a}{N}+\frak x,\frac{b}{N^{1+\sigma}}+\frak t\Big)\\
            &=\int_{\substack{\frak x\in[0,N^{-1}]\\\frak t\in[0,N^{-2}]}}\#\mathcal{A}^{\frak x,\frak t}\\
            &=\int_{\substack{\frak x\in[0,N^{-1}]\\\frak t\in[0,N^{-2}]}}N^{-\sigma}\mu(\{x\in\Q_p^2:|g^{\frak x,\frak t}(x)|>\alpha\}).
            \end{split}
        \end{equation*}
        Applying our uniform upper bound, we conclude.
    \end{proof}

	\bibliography{main}
    \bibliographystyle{plain}

\end{document}